\providecommand{\pgfsyspdfmark}[3]{}
\documentclass[journal, web]{article} 
\usepackage{color}
\usepackage{cite}
\usepackage{amsmath,amssymb,amsfonts,mathabx}  
\usepackage{graphicx,wrapfig}
\usepackage{caption}
\usepackage{subcaption}  

\usepackage{tikz}
\usetikzlibrary{arrows}
 
 \usepackage{enumitem}
\usepackage{upgreek}

  \allowdisplaybreaks
  \makeatletter
   \let\NAT@parse\undefined
  \makeatother
 \usepackage[pdftex, pdfborderstyle={/S/U/W 0}]{hyperref}

\usepackage{comment}

 \newtheorem{thm}{Theorem}
 \newtheorem{exe}{Example}
 \newtheorem{rem}{Remark}
 \newtheorem{lem}{Lemma}
 \newtheorem{prop}{Proposition}
 \newtheorem{claim}{Claim}
 \newtheorem{defin}{Definition}
 \newtheorem{ass}{Assumption}
 \newtheorem{cor}{Corollary}

 \newenvironment{theorem}[1][\,\!]{\ignorespaces\begin{thm}[#1]\rm }{\hfill \mbox{\footnotesize$\square$} \end{thm}}
 \newenvironment{example}{\ignorespaces\begin{exe}\rm }{\hfill \mbox{\footnotesize$\square$} \end{exe}}
 
 \newenvironment{proposition}{\begin{prop}}{\hfill \mbox{\footnotesize$\square$} \end{prop}}
 \newenvironment{remark}{\begin{rem}}{\hfill $\bullet$ \end{rem}}
 \newenvironment{lemma}{\begin{lem}}{\hfill  \mbox{\footnotesize$\square$} \end{lem}}
 \newenvironment{definition}{\vskip 3pt\begin{defin}}{\vskip 3pt
                                          \end{defin}}
 \newenvironment{assumption}{\vskip 3pt\begin{ass}}{\end{ass}\vskip 3pt}
 \newenvironment{proof}{\it Proof. \rm }{\hfill \mbox{\footnotesize$\blacksquare$}}
 \newenvironment{corollary}{\begin{cor}}{\hfill \mbox{\small$\square$} \end{cor}}

\newcommand{\dom}{\mathop{\rm dom}\nolimits}

\title{\LARGE \bf Hybrid Persistency of Excitation in Adaptive Estimation for Hybrid Systems}

\author{A.  Saoud \quad M.  Maghenem \quad   A.  Lor\'{\i}a  \quad R. G. Sanfelice  
\thanks{A. Saoud is with CentraleSupelec, University Paris-Saclay, Gif-sur-Yvette,  France, and the School of Computer Science at Mohammed VI Polytechnical University, Benguerir, Morocco (e-mail: adnane.soud@centralesupelec.fr).  
M. Maghenem is with GIPSA-Lab,  CNRS,  and University of Grenoble Alpes,  Grenoble,  France (e-mail: mohamed.maghenem@cnrs.fr);  
A.  Lor\'ia is with L2S,  CNRS,   91192 Gif-sur-Yvette,  France (e-mail: antonio.loria@cnrs.fr); 
R.  G.  Sanfelice is with the Dept. of Electrical and Computer Engineering, University of California,  Santa Cruz,  CA,  USA (e-mail:ricardo@ucsc.edu). The work of M. Maghenem and A. Lor\'{\i}a was supported by the French ANR via project HANDY, contract number ANR-18-CE40-0010, and the ANR PIA funding: ANR-20-IDEES-0002. }}

\begin{document}
\maketitle 

\begin{abstract} 
We propose a framework of stability analysis  for a class of linear non-autonomous hybrid systems,  with solutions evolving in continuous time governed by an ordinary differential equation and undergoing instantaneous changes  governed by a difference equation.  Furthermore, the jumps may also be triggered by exogeneous hybrid signals.  The proposed framework builds upon a generalization of notions of \textit{persistency of excitation} (PE) and \textit{uniform observability} (UO), which we redefine to fit the realm of hybrid systems. Most remarkably we propose for the first time in the literature a definition of hybrid persistency of excitation. Then, we establish conditions, under which, hybrid PE implies hybrid UO and, in turn,  uniform exponential stability (UES) and input-to-state stability (ISS). Our proofs rely on an original statement for hybrid systems, expressed in terms of $\boldsymbol{\mathcal L_p}$ bounds on the solutions. We also demonstrate the utility of our results on generic adaptive estimation problems. The first one concerns the so-called \textit{gradient systems}, reminiscent of the popular \textit{gradient-descent} algorithm.  The second one pertains to  the design of adaptive observers/identifiers for a class of hybrid systems that are nonlinear in the input and in the output, and linear in the unknown parameters. In both cases,  we illustrate through meaningful examples that the proposed hybrid framework succeeds in scenarii where the classical purely continuous- or discrete-time counterparts fail. 
\end{abstract}

\section{Introduction}

Persistency of excitation 
\cite{ASTBOH}, roughly speaking, is the property of a function of time that consists in the function's energy {\it never} vanishing.  Mathematically, the PE property may be expressed in various forms, depending,  {\it e.g.}, on whether its scalar argument is considered as a real or integer variable,  that is,  on whether the function is evolving in continuous or discrete time.  Over five decades, several definitions of PE have been proposed, in various contexts, to guarantee different stability properties. For linear time-varying systems, some PE properties guarantee uniform (in the initial time) exponential stability \cite{AND77} or uniform global asymptotic stability \cite{doi:10.1137/0315002}. With careful handling, which involves replacing some instance of the state in the system's equations with the system's solutions \cite{KHALIL}, PE-based statements tailored for linear systems may also apply to nonlinear systems \cite{al:LTVsystSCL}. In this case,  a solution-dependent PE notion is necessary and sufficient to ensure uniform asymptotic stability.  For particular classes of nonlinear non-autonomous systems forms of solution-independent PE conditions have been proposed, tailored for functions that depend both on time and the state \cite{TCLEETAC02,TCLEEACC03,al:TACDELTAPE}. A non-solution-dependent relaxed PE condition tailored for nonlinear systems is provided in \cite{al:TACDPEMAT}, where it is also showed to be necessary for uniform global asymptotic stability of generic nonlinear non-autonomous systems. 

The classes of systems where the PE property is used include, but are not restricted to, those appearing in problems of identification 
\cite{tao2003adaptive},  adaptive control \cite{Narendra,IOASUN},  model identification \cite{kurdila1995persistency}, learning-based identification \cite{sridhar2022improving},  and state estimation \cite{Besancon2007overview,989154}. For instance,  the so-called gradient systems,  which appear in the context of gradient-descent estimation algorithms  
are among the linear time-varying systems where PE is necessary and sufficient for UES of the origin.  Moreover, convergence rate  estimates \cite{brockett2000rate,al:LTVsystSCL} and strict Lyapunov functions for gradient systems are available in the literature \cite{JP4-SRIKANT-IJC}.  Other forms of sufficient conditions that involve relaxing the PE property for gradient systems, {\it e.g.}, by admitting the excitation to last only over a finite window of time,  have also been investigated for continuous-time systems \cite{chowdhary_IJC2014} as well as for data-driven models \cite{de2019persistency,dorfler2022persistency}.  Relaxed forms of PE are considered for gradient systems in \cite{praly2017convergence}, but these cannot ensure uniform convergence of the estimation errors towards the origin.  

One of the landmark results in the study of PE is that it is equivalent to uniform observability (UO) \cite{CLAN8} for passive systems satisfying structural properties reminiscent of the Kalman-Yacubovich-Popov Lemma \cite{AND77}.  The first results on  stability of the so-called model-reference-adaptive control schemes rely on such a fact \cite{IOASUN,NARANA}.  For nonlinear time-varying systems, there is an equivalence between PE and zero-state detectability \cite{TCLEETAC02}. The PE property is also broadly present in the context of adaptive observer/identifier design,  both for linear and nonlinear systems \cite{besancon2006adaptive,loria2009adaptive,989154,gevers1988robustness,guyader2003adaptive}.  Roughly speaking,  the parameter estimation  strategy relies on injecting external signals into the system, to excite all the modes and render the system observable, uniformly in the initial conditions. 

As it is well established nowadays,  the coexistence of continuous- and discrete-time  phenomena (what we call \textit{hybrid phenomena}) is unavoidable in some scenarios of control systems.  This is the case under the presence of impacts provoking instantaneous changes in the state,  as in manufacturing systems \cite{chryssolouris2013manufacturing},  cyber-physical systems \cite{sanfelice2016analysis} or when combining continuous and discrete state variables,   or in the presence of shocks and reflection-propagation \cite{chen2018mathematics}.   It is also the case under constrained sensing and actuation,  as in power \cite{9794285} and network control systems \cite{hespanha2007survey}.  In the aforementioned situations the solutions have a continuous evolution,  governed by a continuous-time system,  provided that they stay in a set called the \textit{flow set}.  Furthermore, they experience  instantaneous changes,  governed by a discrete-time system,  once they reach a subset called the \textit{jump set}.   The control of these systems may  require the estimation of some parameters that can affect the continuous- or the discrete-time dynamics, but they can also affect the flow and the jump sets.  Hence,  there is a need to extend the PE-based framework to the realm of the general class of hybrid systems.

In this paper, which is the outgrowth of \cite{ACCGradient},  we study PE in the realm of hybrid systems, using the framework of \cite{goebel2012hybrid}.  This framework covers impulsive systems, which are a type of non-autonomous systems that experience jumps under the influence of a piece-wise continuous signal, and not only depending on whether the state trajectory is in the flow or the jump set at a given instant. Our main contribution is the formulation of a property of PE tailored for a class of hybrid systems. The property we define captures, with particular efficacy,  the richness of time-varying piece-wise-continuous signals; richness that cannot be captured otherwise by classical definitions of PE,  defined purely in continuous or discrete time.  For instance,  we show that a hybrid version of the classical  gradient-descent identification algorithm successfully estimates the unknown parameters of a \textit{hybrid input-output plant} in cases where purely continuous- or purely discrete-time algorithms fail. More importantly, we establish that HPE implies a hybrid form of UO for the considered class of linear time-varying hybrid systems.  In turn, we establish UES and ISS under HPE. These statements are presented in Section \ref{sec:main}. In that light,  we stress that other definitions of observability for hybrid systems have been proposed in the literature, {\it e.g.}, \cite{vidal2003observability,vazquez2018observability}, but these are restricted to switched systems. 
Finally, we address the problem of adaptive observer/identifier design for a class of uncertain hybrid systems, which are affine in the unmeasured states and linear in the  unknown parameters. Based on well-known designs of adaptive observers/identifiers in continuous- and discrete-time \cite{besancon2006adaptive,loria2009adaptive,989154,gevers1988robustness,guyader2003adaptive}, we show that a properly constructed hybrid observer/identifier achieves uniform exponential convergence of the observation and estimation errors. Different from \cite{9682794}, where only the identification problem is solved by assuming  either CPE or DPE, our result holds under the relaxed HPE. Finally, we illustrate through a simple but meaningful example of an impact mechanical system, how the proposed  hybrid observer/identifier may supersede its purely continuous- or discrete-time  counterparts.  

In the next section, for completeness, we recall some definitions and notations that pertain to the hybrid-systems framework of \cite{goebel2012hybrid}.

\section{Preliminaries on Hybrid systems}
\label{sec:prel}

After \cite{goebel2012hybrid}, a hybrid dynamical system $\mathcal{H}$ is the combination of a constrained differential equation and a constrained difference equation given by
\begin{align}\label{305}
\mathcal{H}: & \left\{ 
\begin{array}{ccll} 
\hspace{-.7ex} \dot{x} &\hspace{-1ex} = &\hspace{-1ex} F(x) &\  x \in C  \\  
\hspace{-.7ex}    x^+ &\hspace{-1ex} = &\hspace{-1ex} G(x) &\  x \in D, 
\end{array} \right.
\end{align} 
\noindent where  $x \in \mathcal{X} \subseteq \mathbb{R}^{m_x}$ denotes the state variable, $\mathcal{X}$ the state space,  $C \subseteq \mathcal{X}$ and  $D \subseteq \mathcal{X}$ denote the flow and jump sets,  respectively,  and  $F : C \rightarrow \mathbb{R}^{m_x}$ and $G : D \rightarrow \mathbb{R}^{m_x}$ correspond to the flow and jump maps. Solutions to \eqref{305}  consist in functions with hybrid time domain defined as follows.

\begin{definition}[hybrid signal and hybrid arc]
    A {\it hybrid signal} $\phi$ is a function defined on a hybrid time domain denoted $\dom \phi \subset \mathbb{R}_{\geq 0} \times \mathbb Z_{\geq 0}$. The hybrid signal $\phi$ is parameterized by ordinary time $t \in \mathbb{R}_{\geq 0}$ and a discrete counter $j \in \mathbb Z_{\geq 0}$. Its domain of definition is denoted  $\dom \phi$ and is such that, for each $(T,J) \in \dom \phi$, $\dom \phi \cap \left( [0,T] \times \left\{ 0, 1, \ldots, J \right\} \right) = \cup^{J}_{j=0} \left([t_j,t_{j+1}] \times \left\{j\right\} \right)$ for a sequence $\left\{ t_j \right\}^{J+1}_{j=0}$  such that $t_{j+1} \geq t_j$,  $t_0 = 0$,  and $t_{j+1} = T$.  Moreover, if for each $j \in \mathbb{N}$, the function $t \mapsto \phi(t,j)$ is locally absolutely continuous on the interval $I^j:=\left\{t : (t,j) \in \dom \phi \right\}$, then the hybrid signal $\phi$ is said to be a {\it hybrid arc}.
\end{definition}
%
\begin{definition} [Solution to  ${\mathcal{H}}$] 
  \label{def:sols} \index{forward solution to a hybrid system} A hybrid arc $\phi : \dom \phi \to \mathbb{R}^{m_{\phi}}$ is a {\em solution} to ${\mathcal{H}}$ if  $\phi(0,0) \in \textrm{cl}(C) \cup D$; 
\begin{itemize}[topsep=0pt]               
%
\item[(S2)]  for all $j \in {\mathbb Z_{\geq 0}}$ such that $I^j_{\phi} = \left\{t : (t,j) \in \dom \phi \right\}$ has nonempty interior,  
\begin{equation*}
\begin{array}{ll}               
\phi(t,j) \in C & \mbox{for all} ~ t \in \mbox{int}(I^j_\phi), \cr
\dot{\phi}(t,j) = F(\phi(t,j)) & \mbox{for almost all} 
~ t \in I^j_\phi; \cr
\end{array} 
\end{equation*}                        
\item[(S3)] for all $(t,j) \in \dom \phi$ such that $(t,j+1)\in \dom \phi$,
\begin{equation*}       
\begin{array}{l}         
\phi(t,j) \in D, \qquad
\phi(t,j+1) = G(\phi(t,j)). \cr
\end{array}
\end{equation*}   
\ \\[-14mm]              
\end{itemize}     
\end{definition}

A solution $\phi$ to $\mathcal{H}$ is said to be maximal if there is no solution $\psi$ to $\mathcal{H}$ such that $\phi(t,j) = \psi(t,j)$ for all $(t,j) \in \dom \phi$ and $\dom \phi$ is a proper subset of $\dom \psi$. It is said to be nontrivial if $\dom \phi$ contains at least two points. It is said to be continuous if it is nontrivial and never jumps. It is said to be eventually discrete if $T := \sup_t \dom \phi < \infty$ and $\dom \phi \cap (\{T\} \times \mathbb Z_{\geq 0})$ contains at least two points. It is said to be eventually continuous if $J := \sup_j \dom \phi < \infty$ and $\dom \phi \cap (\mathbb{R}_{\geq 0} \times \{J\})$ contains at least two points.  System $\mathcal{H}$ is said to be forward complete if the domain of each maximal solution is unbounded. 

We are interested in sufficient conditions for UES of a closed set $\mathcal{A} \subset \mathcal{X}$ for a hybrid system $\mathcal{H} := (C,F,D,G)$.  This property is defined in terms of the distance of $\phi$ to the set $\mathcal A$, {\it i.e.}, $|\phi|_{\mathcal A} := \displaystyle\inf_{z\in \mathcal A} |\phi-z|$,   where $|\,\cdot\,|$ denotes Euclidean norm,  as follows---cf. \cite{goebel2012hybrid}. 
\begin{definition}[UES]\label{def:ues}
Let the closed subsets $(\mathcal{A}, \mathcal{D}) \subset \mathcal{X} \times \mathcal{X}$. The set $\mathcal{A}$ is said to be UES for $\mathcal{H}$ on $\mathcal{D}$ if there exist $\kappa$ and $\lambda>0$ such that, for each solution $\phi$ to $\mathcal{H}$ starting from $x_o \in \mathcal{D}$ at $(0,0)$,  we have 
\begin{equation} \label{eqUESbis}
|\phi(t,j)|_{\mathcal{A}} \leq \kappa |x_o|_{\mathcal{A}}  e^{{-\lambda(t+j)}} \quad \forall (t,j) \in \dom \phi.
\end{equation}
 If $\mathcal{D} = \mathcal{X}$, we say that the set 
$\mathcal{A}$ is UES for $\mathcal{H}$.
\end{definition}

\section{Integral Characterization of UES} 
\label{sec:ues}
Our first statement is an original  characterization of UES for hybrid systems, in terms of uniform $\mathcal L_p$-integrability conditions. It is reminiscent of \cite[Lemma 2]{al:TACDELTAPE} for continuous-time systems and in \cite{al:DTINTLEM} for discrete-time systems.  However, as the solutions of hybrid systems may flow and jump, we first introduce certain notations related to integration over a hybrid time domain.
\subsubsection*{Hybrid Integral} 
Consider a   function with hybrid domain  $\phi : \dom \phi \rightarrow \mathbb{R}^{n \times n}$ and let $K \in \mathbb{R}_{>0} \cup \{+ \infty\}$ and $(t,j) \in \dom  \phi$.  We use $E^{ \phi}_{t,j,K} \subset \dom  \phi$ to denote the shortest hybrid time domain, starting from $(t,j)$, of length larger or equal than $K$ and contained in $\dom  \phi$.  
Note that if $K$ is finite,  then there exists a unique $(s_K,m_K) \in \dom \phi$,  such that 
\begin{align} \label{domprop}
K \leq (s_K - t) + (m_K - j) < K+1,  
\end{align}
and a unique non-decreasing sequence 
$$ \{t_j, t_{j+1},...,t_{m_K}, t_{m_K + 1} \} ~ \text{with} ~ t_j := t ~ \text{and} ~ t_{m_K + 1} := s_K,$$ 
such that 
\begin{align*} 
 E^{ \phi}_{t,j,K}  :=  
[t_j, t_{j+1}] \times \{ j \} \cup  \cdots  \cup [t_{m_K}, t_{m_K + 1}] \times 
\{ m_K \}.
\end{align*} 
 Thus,  the hybrid integral of $ \phi$ over the domain $E^{ \phi}_{t,j,K}$ is defined as
\[  \hspace{-3pt} \int_{E^{ \phi}_{t,j,K}} \hspace{-6pt} \phi(s, i) d(s,i) := 
\sum\limits_{i=j}^{m_K} \int_{t_{i}}^{t_{i+1}} \hspace{-4pt}
 \phi (s,i) ds + \sum\limits_{i=j}^{m_K - 1}  \phi(t_{i+1},i).   \]
 In particular, for  $K = + \infty$, we have  $s_\infty + m_\infty = 
 +\infty$. 

Akin to the case where signals evolve purely in continuous or discrete time---cf. \cite{DESVID}, given a  function  $\phi$, with hybrid domain starting at $(t_o,j_o)\in \dom \phi$,  we define the hybrid $\mathcal L_p$-norm,  with $p\in [1,\infty)$,  as
\begin{equation}
  \label{lpnorm}  \big| \phi \big|_{\mathcal{A} p}  := \left[\int_{E^{\phi}_{t_o,j_o,\infty}} \hspace{-0.4cm} \big|\phi(s, i) \big|^p_{\mathcal{A}} d(s,i) \right]^{\frac{1}{p}}
\end{equation}
and the hybrid $\mathcal L_\infty$ norm,
\begin{equation}
  \label{linfnorm}  \big| \phi \big|_{\mathcal{A} \infty}  := \sup \left\{ \big| \phi (t,j) \big|_{\mathcal{A}} : (t,j) \in E^{\phi}_{t_o,j_o,\infty} \right\}. 
\end{equation}
In the case that $\mathcal A = \{0\}$ we simply write $|\phi|_p$ and $|\phi|_\infty$.


Then, the following statement generalizes \cite[Lemma 3]{al:LTVsystSCL} to the realm of hybrid systems. 

\begin{theorem}[Hybrid-integral characterization of UES] \label{prop.integ}
Consider the hybrid system $\mathcal{H} := (C,F,D,G)$, as defined in \eqref{305},  and let  $(\mathcal{A}, \mathcal{D}) \subset \mathcal{X} \times \mathcal{X}$ be closed subsets. Assume that there exist $c$ and $p>0$ such that, for each  $\phi$, solution to $\mathcal{H}$ starting from $x_o \in \mathcal{D}$,  we have 
\begin{align} \label{eqinteg}
\max \left\{ \big| \phi \big|_{\mathcal{A} \infty}, \ \big| \phi \big|_{\mathcal{A}p} \right\} \leq c \big| x_o \big|_{\mathcal{A}}.
\end{align}
Then, the set $\mathcal{A}$ is UES on $\mathcal{D}$ and \eqref{eqUESbis} holds with  $ \lambda := \frac{1}{pc^p}$ and  $\kappa := c\exp(1/p)$.
\end{theorem}

\begin{proof}
We first remark that the solutions $\phi$ start at $(0,0)$, so the $\mathcal L$-norms in \eqref{eqinteg} are to be considered on $E^{\phi}_{0,0,\infty}$. Now, following  the proof lines of \cite[Lemma 3]{al:LTVsystSCL}, we note that condition (\ref{eqinteg}) implies that,  for all $(t,j) \in E^{\phi}_{0,0,\infty}$,\\[-1pt] 
\begin{equation}
\label{eqinteg2}
    \sup\big\{\,|\phi(s,i)|_{\mathcal{A}}^p : (s,i) \in E^{\phi}_{t,j,\infty} \,\big\} \leq c^p |\phi(t,j)|_{\mathcal{A}}^p
\end{equation}
and
\begin{equation}
\label{eqinteg3}
    \int_{E^{\phi}_{t,j,\infty}}  \hspace{-0.4cm}  \big|\phi(s, i) \big|^p_{\mathcal{A}} d(s,i) \leq c^p |\phi(t,j)|_{\mathcal{A}}^p. 
\end{equation}
Next, we define the hybrid arc $v:\dom \phi \rightarrow \mathbb{R}_{\geq 0}$ given by 
$$ v(t,j):=\int_{E^{\phi}_{t,j,\infty}}  \hspace{-0.4cm} \big|\phi(s, i) \big|^p_{\mathcal{A}} d(s,i),  $$
and we distinguish the two following cases: for all $t$ such that the solution flows, that is, if $t \in \mbox{int}(I^j_{\phi})$ with $I^j_{\phi} := \{t : (t,j) \in \dom \phi\}$, we have 
\begin{align*}
 \dot{v}(t,j)=\frac{d}{dt}\left[ \int_t^{t_{j+1}} \hspace{-0.4cm} |\phi(s,j)|_{\mathcal{A}}^p ds  \right]  
= -|\phi(t,j)|_{\mathcal{A}}^p  \leq  -\frac{1}{c^p} v(t,j);
\end{align*}
the last inequality follows from (\ref{eqinteg3}). If the solution jumps, that is, for all $(t,j) \in \dom \phi$ such that $(t,j+1) \in \dom \phi$, we have 
\begin{align*}
    v(t,j+1)-v(t,j) = -|\phi(t,j)|_{\mathcal{A}}^p \leq  \frac{-1}{c^p} v(t,j);
\end{align*}
again, the last inequality follows from (\ref{eqinteg3}).  
As a result,  using the comparison principle for hybrid systems---Lemma~\ref{lem:comparison} in the Appendix, while replacing $a$ therein by $\frac{1}{c^p}$,  we conclude that 
\begin{equation} \label{equse2}
    v(t,j) \leq e^{-\frac{t+j}{c^p}} v(0,0). 
\end{equation}
Now, we consider a parameter $K>0$ and note that, for each $(t,j) \in E^\phi_{0,0,\infty}$,
\begin{align*}
 & v(t,j) = \int_{E^{\phi}_{t,j,\infty}} \hspace{-3pt} \big|\phi(s, i) \big|^p_{\mathcal{A}} d(s,i) \geq \int_{E^{\phi}_{t,j,K}} 
\hspace{-3pt} \big|\phi(s, i) \big|^p_{\mathcal{A}} d(s,i)\\ 
& \qquad \geq  \sum\limits_{i=j}^{m_K} \int_{t_{i}}^{t_{i+1}} \hspace{-3pt}
 \big|\phi(s, i) \big|^p_{\mathcal{A}} ds + \sum\limits_{i=j}^{m_K-1}  \big|\phi(t_{i+1}, i) \big|^p_{\mathcal{A}}\\ 
& \qquad \geq 
\frac{1}{c^p}\Bigg[\sum\limits_{i=j}^{m_K} \int_{t_{i}}^{t_{i+1}} \hspace{-0.6cm}
 \sup\Big\{\big|\phi(\tau, k) \big|^p_{\mathcal{A}}: (\tau,k) \in E^{\phi}_{t,j,K}\Big\}ds \Bigg] \\ 
& \qquad +  \frac{1}{c^p}\Bigg[\sum\limits_{i=j}^{m_K-1}  \sup\Big\{\big|\phi(\tau, k) \big|^p_{\mathcal{A}}: 
 (\tau,k) \in E^{\phi}_{t,j,K}\Big\} \Bigg] \\ 
& \qquad \geq \frac{s_{K} - t +m_K - j}{c^p}   \sup\Big\{\big|\phi(\tau, k) \big|^p_{\mathcal{A}}: (\tau,k) \in E^{\phi}_{t,j,K}\Big\}  \\ & \qquad 
\geq \frac{K}{c^p}  \big|\phi(s_{K},m_K) \big|^p_{\mathcal{A}}, 
\end{align*}
where the last inequality comes from (\ref{domprop}). Then, we define $K:=c^p$ and we use (\ref{eqinteg2}) 
and \eqref{equse2} to conclude that, for each $(t,j) \in E^\phi_{0,0,\infty}$, 
\begin{align*}
\big|\phi(s_K, m_K) \big|^p_{\mathcal{A}} & \leq v(t,j)  
\leq  e^{-\frac{t+j}{c^p}} v(0,0)  \leq c^p e^{-\frac{t+j}{c^{p}}} 
\big|\phi(0,0) \big|^p_{\mathcal{A}}.
\end{align*}
The last inequality implies that, for each $ (s,i) \in \dom \phi \backslash  E^\phi_{0,0, K}$,
$$ \big|\phi(s,i) \big|_{\mathcal{A}} \leq c  e^{\frac{K}{pc^{p}}} e^{-\frac{s + i}{pc^{p}}} \big|\phi(0,0) \big|_{\mathcal{A}}.  $$
On the other hand,  for each $ (s,i) \in   E^\phi_{0,0, K}$,
\begin{align*}  
\big|\phi(s,i) \big|_{\mathcal{A}} & \leq 
c \big|\phi(0,0) \big|_{\mathcal{A}} 
 \leq   c  e^{\frac{K}{pc^{p}}}   e^{-\frac{s+i}{pc^{p}}}     
  \big|\phi(0,0) \big|_{\mathcal{A}}.
   \end{align*}
The statement follows.
\end{proof}

\section{UES and ISS for time-varying hybrid systems} \label{sec:main}

Consider the non-autonomous hybrid  system of the form
\begin{align}\label{332}
\mathcal H' : & \left\{ 
\begin{array}{ccll} 
\hspace{-.7ex} \dot \zeta  &\hspace{-1ex} = &\hspace{-1ex} F'(\zeta,t,j) &  
t \in \text{int}(I^j_A)
\\ [2pt]  
\hspace{-.7ex} \zeta^+ & \hspace{-1ex} = &\hspace{-1ex} G'(\zeta,t,j) & (t,j), (t,j+1) \in \dom A,
\end{array} \right.
\end{align}
\noindent with state $\zeta \in \mathbb{R}^{m_\zeta}$, $F'$, $G': \mathcal X \to \mathbb R^{m_\zeta}$, $\mathcal X :=  \mathbb R^{m_\zeta} \times \dom A$, and such that $A$ is a hybrid signal whose domain is $\dom A$ and $I^j_A := \{t : (t,j) \in \dom A\}$. $A$ may be an exogenous hybrid signal or may also depend on the system's hybrid trajectories---see Section \ref{sec:egs} for examples. Then, the solutions to \eqref{332} are hybrid arcs whose domain is a subset of $\dom A$. That is, the solutions of \eqref{332} jump whenever $A$ jumps.

To study the behavior of the solutions to \eqref{332}, we recast it in the form of \eqref{305},  by including the hybrid time as a bi-dimensional state variable. That is,  defining  $x:=[\xi^\top \ p\ \, q]^\top$,  system \eqref{332} can be rewritten as 
\begin{align}
\hspace{-0.2cm} \mathcal H : & \left\{ 
\begin{array}{rcll} 
\hspace{-0.2cm}
\begin{bmatrix} \,\dot \xi\, \\ \dot p \\ \dot q \end{bmatrix} &\hspace{-1ex} = &\hspace{-1ex} 
\begin{bmatrix} F'(\xi, p,q) \\ 1 \\ 0 \end{bmatrix} & x \in C
\\[17pt]
\hspace{-0.2cm}
\begin{bmatrix} 
\xi^+  \\ p^+ \\ q^+ 
\end{bmatrix} &\hspace{-1ex} = &\hspace{-1ex} 
\begin{bmatrix} 
G'(\xi,p,q) \\ p \\ q + 1 \end{bmatrix} & x \in D,
\end{array} \right.
\label{597}
\end{align}
where the flow and jump sets are,  respectively,  defined as 
$C:= \mathcal X$ and 
$D := \{ x  \in \mathcal X :  (p,q+1) \in \dom A \}$.  Then, a solution $\zeta$ to \eqref{332},  starting from the initial condition $\zeta_o  \in \mathbb{R}^{m_\zeta}$ at $(t_o,j_o) \in  \dom A$, must coincide with a solution $\phi$ to \eqref{597},  starting from the initial condition $(\xi_o, t_o, j_o)$ at $(0,0)$.   In this case,  we have $p(t,j) = t + t_o$ and $q(t,j) = j+j_o$ for all $(t,j) \in \dom \phi$. 
We use this fact in what follows of the paper to analyze time-varying hybrid systems in the form of \eqref{332}.

\begin{remark}\label{rmk3}
If the set $\{ (\xi,p,q) \in \mathcal X : \xi = 0 \} $ is UES for 
$\mathcal H$,  as per Definition \ref{def:ues},  then the origin $\{\zeta = 0\}$ is UES for $\mathcal H'$,  that is, every solution $\zeta$,  starting at $(t_o,j_o)$ from $\zeta_o$,  satisfies
\begin{equation} 
  \label{608} |\zeta(t,j)| \leq \kappa |\zeta_o| e^{{-\lambda(t+j-t_o-j_o)}} \quad \forall (t,j) \in \dom \zeta,
\end{equation}
with $\kappa$ and $\lambda$ independent of $(t_o,j_o)$. 
\end{remark}

\subsection{Problem formulation and standing hypotheses}

In the sequel, we focus on perturbed non-autonomous hybrid systems of the form---cf. Eq. \eqref{332},
\begin{equation}
  \label{606} 
\hspace{-0.2cm}  \mathcal H_\nu' :  \left\{ 
\begin{array}{ccll} 
\hspace{-.9ex} \dot{\zeta}  &\hspace{-2ex} = & \hspace{-2ex} -A(t,j)\zeta + \nu(t,j)
&   \hspace{-13ex} t \in  \text{int}(I^j_A)  
\\
\hspace{-.7ex}    \zeta^+ &\hspace{-2ex} = & \hspace{-2ex} [\,I_{m_\zeta} -B(t,j)\,]\zeta 
+ \nu(t,j) & \\ 
&&   \hspace{10ex}  (t,j), (t,j+1) \in \dom A,
\end{array} \right. 
 \end{equation}
where $A$ and $B$ (are assumed to) have the same hybrid time domain,  that is,  $A$ and $B: \dom A \rightarrow \mathbb{R}^{m_\zeta \times m_\zeta}$. Furthermore,  $\nu : \dom A \rightarrow \mathbb R^{m_\zeta}$ is an external hybrid perturbation. The index $_\nu$ in $\mathcal H_\nu$ is to distinguish the system in \eqref{606} from the unperturbed dynamics resulting from setting $\nu\equiv 0$.
\begin{remark}\label{rmk5}
This class of systems is important as it covers a number of interesting cases that appear in adaptive estimation. 
 For instance,  when $\nu\equiv 0$ and $A(t,j)$ and $B(t,j)$ are both symmetric and positive semidefinite,  system \eqref{606} generalizes the so called \textit{gradient system}, studied both in continuous and discrete time in the context of identification  \cite{AND77,brockett2000rate} and multi-agent systems \cite{9312975,JP4-SRIKANT-IJC}. The functions $A$ and $B$ may come from expressing outputs and inputs along solutions; namely,  for a system $\dot z = A_z(y,u)z$,  we let $A(t,j) := A_z(y(t,j),u(t,j))$.  This artifice is commonly used to analyze some nonlinear  observers  \cite{Besancon2007overview,989154}. See also Section \ref{sec:egs}.
\end{remark}

In what follows, we investigate sufficient conditions for the origin 
$\{\zeta =0\}$ to be UES for $\mathcal H'_0$  (that is, \eqref{606} with $\nu\equiv 0$) and for the system $\mathcal H'_\nu$ to be  ISS with respect to $\nu\not\equiv 0$. For hybrid systems, ISS means that there exist a class $\mathcal{K}_\infty$ function $\alpha$ and a class $\mathcal{K} \mathcal{L}$ function $\beta$ such that,  for each solution $\zeta$ to $\mathcal H'_\nu$,  starting from $\zeta_o \in \mathbb{R}^{m_\zeta}$,  at $(t_o,j_o) \in \dom A$,   we have  
\begin{align*} 
 |\zeta(t,j)|_{\mathcal{A}} & \leq \beta (|\zeta_o|_{\mathcal{A}}, t+j-t_o-j_o)  
\\ &
+ \alpha \left( \sup \left\{ |\nu(s,i)| :  (s,i) \in \dom \phi \backslash  E^\zeta_{t,j,\infty}  \right\} \right)   
\end{align*}
for all $(t,j) \in E^\zeta_{t_o,j_o,\infty}$.

We solve these problems under two standing hypotheses reminiscent of others that are  common in the context of  continuous-  or discrete-time systems. The first one essentially guarantees boundedness of the solutions and uniform global stability of the origin $\{ \zeta = 0 \}$  for $\mathcal H'_0$. Roughly, for  $\mathcal H'_{0}$, we require the existence of a Lyapunov function with negative semidefinite derivative along flows and non-increasing over jumps. The second Assumption imposes uniform boundedness of the matrices $A$ and $B$.  

\begin{assumption} [Lyapunov (Non-Strict) Inequalities] 
\label{ass:UGS}
There exists a symmetric matrix 
$P : \dom A \rightarrow \mathbb{R}^{m_\zeta \times m_\zeta}$ and constants $p_1,~p_2>0$ such that
$p_1 \leq  |P|_\infty \leq p_2 $.  
Furthermore, there exist symmetric positive  semi-definite matrices $Q_c,~Q_d : \dom P \rightarrow \mathbb{R}^{m_\zeta \times m_\zeta}$ such that, for all $t\in \mbox{int}(I^j_A)$,  
\begin{equation}
  \label{484}  \dot{P}(t,j) - A(t,j)^\top P(t,j) -  P(t,j) A(t,j) \leq -Q_c(t,j), 
\end{equation}
and,  for all $(t,j) \in \dom A$ such that $(t,j+1) \in \dom A$,  
\begin{align}
\nonumber
  [I_{m_\zeta} - B(t,j) ]^\top  P(t,j+1) [I_{m_\zeta} & - B(t,j)] - P(t,j)\qquad \\ 
 &\hfill  \leq   - Q_d(t,j). 
\label{491}
\end{align}
\end{assumption}

\begin{assumption} [Uniform Boundedness] \label{ass:linf}
There exist $\bar{A}$, $\bar{B}>0$ such that $|B|_\infty \leq \bar{B}$ and  $|A|_\infty \leq \bar{A}$.
\end{assumption}


In addition to Assumptions \ref{ass:UGS} and \ref{ass:linf},  we investigate the role of HUO and HPE. 

\subsection{UES and ISS under HUO}

Consider the linear system $\mathcal H'_0$, {\it i.e.}, \eqref{606} with $\nu\equiv 0$.  We introduce the hybrid transition matrix $\mathcal{M} : \dom A \times \dom A \rightarrow \mathbb{R}^{m_\zeta \times m_\zeta}$ such that,  for each $((t,j), (t_o,j_o)) \in \dom A \times \dom A$,  the solution $\zeta$ starting from $\zeta_o$ at $(t_o,j_o)$ satisfies 
\begin{equation}
  \label{560}  \zeta(t,j) = \mathcal{M}((t,j),(t_o,j_o)) \zeta_o.
\end{equation}
The hybrid transition matrix  $\mathcal{M}$ is the solution to the system  
\begin{subequations} 
\label{htm}
\begin{align} 
\label{htm:a} 
& \hspace{-0.1cm} \dot{\mathcal{M}}((t,j), (t_o,j_o)) =   -A(t,j) \mathcal{M}((t,j), (t_o,j_o)) ~~~ t \in I^j_A  
 \\ 
\nonumber  
& \hspace{-0.1cm} \mathcal{M}((t,j+1),(t_o,j_o))  =  \left[I_{m_\zeta} - B(t,j) \right] \mathcal{M}((t,j),(t_o,j_o)) 
\\ \label{htm:b}     & 
\hspace{8.em}  
(t,j), (t,j+1) \in \dom A
\\
\label{htm:c} 
& \hspace{-0.1cm} \mathcal{M}((t_o,j_o), (t_o,j_o))  = \, I_{m_\zeta}.
\end{align}
\end{subequations}

Then,  we introduce the following property.

\begin{definition}[HUO] \label{def:HUO}
The pair $(A,B)$ satisfying Assumption \ref{ass:UGS} is HUO if there exist $K$, $\mu > 0$ such that, for each $(t_o,j_o) \in \dom A$, 
\begin{align} \label{eq:HUO}
\int_{E^A_{t_o,j_o,K}}  \hspace{-20pt} \mathcal{M} 
\left((s,j), (t_o,j_o)\right)^\top & \Phi(s,j) \mathcal{M}\left((s,j), (t_o,j_o) \right) d(s,j) \nonumber  \\ & \geq \mu I_{m_\zeta}, 
\end{align}

\noindent where  $\Phi : \dom \Phi \rightarrow \mathbb{R}^{m_\zeta \times m_\zeta}$, with $\dom \Phi = \dom Q_c = \dom Q_d$,  is given by 
\begin{align} \label{eqAab}
\Phi(t,j) := 
\left\{ 
\begin{array}{ll} 
\!\! Q_c(t,j) & \text{if} ~~  t \in \text{int}(I^j_A) 
\\ 
\!\! Q_d(t,j) & \text{otherwise}.
\end{array}
\right.
\end{align}
\end{definition}


\begin{remark} \label{rmk6}
Following up on Remark \ref{rmk5},  we note that particular instances of HUO pairs pertain to 
multi-variable systems,  where $Q_c(t,j)$ and $Q_d(t,j)$ result from designing a hybrid input.   Thus,  the required HUO property may be induced (by design).  
\end{remark}

\begin{theorem}[HUO implies UES and ISS]  \label{HUO=>UES}
If for the hybrid system $\mathcal H'_0$, defined by \eqref{606} with $\nu\equiv 0$,   Assumptions \ref{ass:UGS} and \ref{ass:linf} hold,  and the pair $(A,B)$ is HUO, then the origin $\{\zeta=0\}$ is UES,  and $\mathcal H'_\nu$ is ISS with respect to $\nu$.
\end{theorem}

\begin{proof}
The stability of the origin $\{\zeta=0\}$ for $\mathcal H'_0$ may be analyzed using the framework described in Section \ref{sec:prel},  by rewriting the system as one that is time-invariant, of the form \eqref{597},  with flow and jump maps
\begin{subequations}\label{616}
\begin{align}
\label{616a} 
 F_\nu(x) \color{black} &:= [-  \xi^\top A(p,q)^\top + \nu^\top  \quad 1 \quad 0]^\top\\
\label{616b} 
 G_\nu(x)  &:= [ \xi^\top[ I_{m_\zeta}  - B(p,q)]^\top +\, \nu^\top   \quad p  \quad q+1 ]^\top, 
\end{align}
\end{subequations}
state $x := (\xi, p, q) \in \mathcal{X} :=  \mathbb{R}^{m_\zeta} \times \dom A$, and flow and jump sets defined by  $C := \mathcal{X}$ and $D := \{ x \in \mathcal{X} : (p,q+1) \in \dom A \}$, respectively.  In particular,  after Remark \ref{rmk3},  the UES bound \eqref{608} holds for $\mathcal H'_0$ if the set
\begin{align} \label{eqsetAG}
\mathcal{A} := \{ x \in \mathcal X : \xi = 0 \} 
\end{align} 
is UES (as per Definition \ref{def:ues}) for $\mathcal H$, defined by  \eqref{597}, \eqref{616}, $C$, and $D$ as defined above. Thus, to prove the first item we use Theorem \ref{prop.integ} and this equivalent time-invariant representation of $\mathcal H'_0$. That is, we explicitly compute $c>0$ such that,  along each solution $\phi$ to \eqref{597}-\eqref{616},   with $\nu \equiv 0$ and starting from $x_o := (\xi_o, t_o,j_o)$,  at $(0,0)$, it holds that 
\begin{equation}
  \label{736}  \max \left\{ \big| \phi \big|^2_{\mathcal{A} \infty},  \big| \phi \big|^2_{\mathcal{A}2} \right\} \leq c \big| x_o \big|^2_{\mathcal{A}},
\end{equation}
where $\mathcal A$ is defined in \eqref{eqsetAG}. 

To that end, we introduce the Lyapunov function candidate
\begin{equation} \label{eqn:lyap}
V(x) := \xi^{\top} P(p,q ) \xi, 
\end{equation}
where $P$ is introduced in Assumption \ref{ass:UGS}. Furthermore, after the latter, we have 
\[ \langle\nabla V(x), F_0(x)  \rangle   \leq -\xi^{\top} Q_c(p,q ) \xi,  \]
for each  $x\in C$,  while 
\[ V( G_0(x) )-V(x) \leq  -\frac{1}{2} \xi^{\top} Q_d(p,q) \xi,  \]
for each $x \in D$. Therefore, after \eqref{eqAab},   along the maximal solution $\phi$, we have 
\[ \dot{V}(\phi(t,j)) \leq - \xi(t,j)^\top \Phi(p(t,j),q(t,j)) \xi(t,j), \]
for all $t \in \mbox{int}(I^j_A)$, while
\begin{align*}
 V(\phi(t,j+1)) - & V(\phi(t,j)) \\
& \leq - \xi(t,j)^\top \Phi(p(t,j),q(t,j)) \xi(t,j)
\end{align*}
for all  $(t,j) \in \dom \phi$ such that  $(t,j+1) \in \dom \phi$. Thus, using the fact that $Q_c$ and $Q_d$ are positive definite from Assumption \ref{ass:UGS}, it follows that 
\[ V(\phi(t,j)) \leq  V(\phi(0,0)) \qquad \forall (t,j) \in E^\phi_{0,0,\infty}, \]
which implies that, for each $(t,j) \in E^\phi_{0,0,\infty}$,  we have   
\[ p_1 \big| \xi(t,j) \big|^2 \leq V(\phi(t,j)) \leq  V(\phi(0,0 )) \leq p_2 \big| \xi_o \big|^2. \]
Finally, since $|\xi| = \big| \phi \big|_{\mathcal{A}}$, we conclude that
\[ \big| \phi \big|^2_{\mathcal{A} \infty} \leq \frac{p_2}{p_1} \big| \phi(0,0 ) \big|^2_{\mathcal{A}}.  \]
This establishes the first bound in \eqref{736}. 

Next, we compute the second bound. To that end, we follow the proof steps of  \cite[Proposition 1]{loria2009adaptive}.  
Let the HUO property generate $K >0$ and, for each $(t,j)\in\dom \phi$, a unique pair $(s_K, m_K)\in \dom \phi$ satisfying \eqref{domprop}. We have 
\begin{align*}
 V(\phi(t,j))\! & - \!  V(\phi(s_K,m_K)) \geq  \int_{E^{\phi}_{t,j,K}}  \hspace{-0.5cm} \xi(s,i)^\top \Phi(s,i) \xi(s,i) d(s,i).
\end{align*}

The hybrid arc $\xi(s,i)$, with $(s,i) \in E^{\phi}_{t,j,K}$, starting at $(t, j)$ coincides with $\zeta(s,i)$ starting at $(t + t_o, j + j_o)$. Therefore, the relation
\begin{align*}
& \zeta(s + t_o, i + j_o)  =   
\\ & 
\mathcal M\big((s \!+\! t_o, i\! +\! j_o),(t \!+\! t_o, j \!+\! j_o)\big) 
\zeta\big((t \!+\! t_o, j \!+\! j_o)\big),
\end{align*}
which holds under \eqref{560}, implies that
\begin{align*}
& \xi(s,i)  = 
\overline{\mathcal M}((s,i),(t,j))
\xi(t,j),
\end{align*}
where 
$\overline{\mathcal M}((s,i),(t,j)) := \mathcal M((s + t_o, i + j_o),(t + t_o, j + j_o))$. 
As a result, we obtain 
\begin{align*}
& V(\phi(t,j)) - V(\phi(s_K,m_K)) \geq \\ & \xi(t,j)^\top \int_{E^{\phi}_{t,j,K}}  \hspace{-14pt} \overline{\mathcal M}((s,i),(t,j))^\top  \Phi(s,i)  \widebar{\mathcal M}((s,i),(t,j)) d(s,i) 
 \\ & \times  \xi(t,j) \geq \mu \big| \xi(t,j) \big|^2.
\end{align*}
Next, we use the fact that 
$$ \big| \xi(t,j) \big|^2 \geq \frac{p_1}{p_2} \big| \xi(s_K,m_K) \big|^2, $$
to obtain 
\begin{align*}
 V(\phi(t,j)) - V(\phi(s_K,m_K))  \geq  \mu \frac{p_1}{p_2} \big| \xi(s_K,m_K) \big|^2.
\end{align*}
Now, integrating on both sides over $E^{\phi}_{0,0,\infty}$, and using the fact that
$$ |\xi(s,i)|^2 = |\phi(s, i)|^2_{\mathcal{A}} \qquad \forall (s,i) \in E^\phi_{0,0,\infty}, $$ 
we obtain 
\begin{align*}
\int_{ E^{\phi}_{0,0,K}} \hspace{-3pt} V(\phi(s,i)) d(s,i) & \geq \frac{p_1 \mu}{p_2}  
\int_{ E^{\phi}_{0,0,\infty}}  \hspace{-3pt} \big|\phi(s, i) \big|^2_{\mathcal{A}} d(s,i)
\\ &
- \frac{p_1 \mu}{p_2} \int_{ E^{\phi}_{0,0,K}} 
\hspace{-3pt}
\big|\phi(s, i) \big|^2_{\mathcal{A}} d(s,i),
\end{align*}
which, in turn, implies that
\begin{align*}
\int_{E^{\phi}_{0,0,\infty}} \hspace{-4pt} \big|\phi(s, i) \big|^2_{\mathcal{A}} & d(s,i) 
  \leq \frac{p_2}{p_1 \mu} \int_{E^{\phi}_{0,0,K}} \hspace{-4pt} V(\phi(s,i)) d(s,i)  
\\ & \hspace{4em}+ \int_{E^{\phi}_{0,0,K}} 
\hspace{-4pt}
\big|\phi(s, i) \big|^2_{\mathcal{A}} d(s,i) 
\\ &
\qquad \leq (K+1) \left[ \frac{p_2^2 }{p_1 \mu} + \frac{p_2 }{p_1} 
 \right] \big|\phi(0,0) \big|^2_{\mathcal{A}}. 
\end{align*}
This completes the proof of UES.

Next, to prove ISS of $\mathcal H'_\nu$ with respect to $\nu$,  we introduce the function $W : \mathcal X \to \mathbb R_{\geq 0}$ given by 
$$ W(x) := \xi^\top  \mathbb P(p,q) \xi,  $$  
 where 
$$ \mathbb P(p,q) := \int_{E^A_{p,q,\infty}} \hspace{-5pt}  \left[ \mathcal{M}((s,i),(p,q))  \mathcal{M}((s,i),(p,q))^\top \right] d(s,i),  $$
and we prove the following claim. 

\begin{claim}
Under UES of the set $\mathcal{A}$ for $\mathcal{H}'_0$, the fact that $\mathcal{M}((p,q),(p,q)) = I_{m_\zeta}$---see \eqref{htm:c}, and $A$ being bounded, we conclude that there exist $p_M \geq p_m > 0$ such that 
\begin{align} \label{eqpropre}   p_m I_{m_\zeta}  \leq \mathbb P(p,q) \leq p_M I_{m_\zeta}  \qquad \forall (p,q) \in \dom A.  
  \end{align}
\end{claim}
\vskip 7pt
\begin{proof}
The upper bound in \eqref{eqpropre} is a straightforward consequence of UES
of the set $\mathcal{A}$ for $\mathcal{H}'_0$ and the definition of the hybrid transition matrix $\mathcal{M}$. 

To prove the lower bound, we consider the following complementary cases: 
\begin{enumerate}
\item If $(p,q) \in \dom A$ and $(p,q+1) \in \dom A$, we conclude that
$$ \mathbb{P}(p,q) = I_{m_\zeta} + \mathbb{P}(p,q+1) \geq I_{m_\zeta}. $$ 
\item If $([p,p+1],q) \in 
\dom A$, we conclude that  
\begin{align*} 
\mathbb P(p,q) & \geq \int^{p+1}_{p} \hspace{-5pt}  \left[ \mathcal{M}((s,q),(p,q)) \mathcal{M}((s,q),(p,q))^\top \right] ds.   
\\ & 
\geq \int^{p+1}_{p}  e^{(-2 \bar{A} (s-p))}  ds I_{m_\zeta} = \int^{1}_{0} e^{-2 \bar{A} s} ds I_{m_\zeta}.
\end{align*}
To obtain the second inequality, we used \eqref{htm:a} and boundedness of the matrix $A$. 

\item If $([p,p+\lambda],q) \in \dom A$ and $(p+\lambda,q+1) \in \dom A$, for some $\lambda \in (0,1)$. In this case, we have 
\begin{align*} 
\mathbb P(p,q) & \geq \mathcal{M}((p+\lambda,q),(p,q))  \mathcal{M}((p+\lambda,q),(p,q))^\top
\\ & \geq  e^{-2 \bar A \lambda} I_{m_\zeta} \geq e^{-2 \bar A} I_{m_\zeta} . 
\end{align*}
\end{enumerate}\vskip -13pt
\end{proof}
  
Furthermore, using \eqref{htm:a} and \eqref{htm:b} along maximal trajectories, we conclude that  
$$ \dot{\mathbb P}(p,q) =  - I_{m_\zeta} - A(p,q)^\top \mathbb P(p,q) - \mathbb P(p,q) A(p,q) $$    
for all $t\in \mbox{int}(I^j_A)$,  while 
$$  \mathbb P(p,q) = I_{m_\zeta} + B(p,q)^\top \mathbb P(p,q+1) B(p,q) $$
for all  $(p,q) \in \dom A$ such that  $(p,q+1) \in \dom A$. 
Therefore, 
\begin{align*}
\langle \nabla W(x), F_0(x)  \rangle & = - |\xi|^2, 
\qquad  \forall\, x \in C\\
 W(G_0(x))  - W (x) &= -  |\xi|^2,\qquad  \forall\, x \in D, 
\end{align*}
where $F_\nu$ and $G_\nu$ are defined in \eqref{616}. 
In turn, 
\begin{align*}
& \hspace{-0.8cm}  \langle \nabla W(x),   F_\nu(x)  \rangle   =
- | \xi |^2  + 2  \xi^\top\mathbb P(p,q)  \nu
\\ & \leq - \frac{1}{2} |\xi|^2 +  4 p_M  
|\nu|^2 \leq - \frac{1}{2 p_M}W(x) +  4 p_M  |\nu|^2,
\end{align*}
for all  $x \in C$, where the first inequality follow from Young's inequality, and 
\begin{align*}
W(G_\nu(x) ) - W (x) = -  |\xi|^2  +  W(G_\nu(x)) -W(G_0(x)),
\end{align*}
for all  $x \in D$. On the other hand, after Assumption \ref{ass:linf}, there exists $b_M'$ such that $|I_{m_\zeta} - B(p,q)| \leq b_M'$ for all $(p,q)\in \dom A$, one gets using Young's inequality that 
\begin{align*}
W(G_\nu(x))  -  W(G_0 (x))  
 & = 2 \nu^\top \mathbb P (p,q+1) [I_{m_\zeta} - B(p,q)] \xi  \\ 
& \quad + \nu^\top \mathbb P (p,q+1) \nu \\ 
& \leq c |\nu|^2 + \frac{1}{2}  |\xi|^2,
\end{align*}
where $c:= p_M + p_M^2 b_M'^2$. In turn, 
\begin{align*}
W(G_\nu(x)) -W(x) = - \frac{1}{2 p_M} W(x)  + c|\nu|^2. 
\end{align*}
So, along the system's trajectories, we have 
\[ \dot W(\phi(t,j)) \leq - \frac{1}{2 p_M}W(\phi(t,j)) +  4 p_M  |\nu(t,j)|^2, \]
for all  $t\in \mbox{int}(I^j_\phi)$, and  
\begin{align*}
\hspace{-0.4cm}W(\phi(t,j+1)) & -W(\phi(t,j)) 
\\ &
\leq - \frac{1}{2 p_M} W(\phi(t,j))  + c |\nu(t,j)|^2,  
\end{align*} 
for all   $(t,j) \in \dom \phi$ such that  $(t,j+1) \in \dom \phi$.
Finally,  we introduce 
$$  \bar\nu := \sup \left\{ |\nu(s,i)| :  (s,i) \in \dom \phi \backslash  E^\phi_{t,j,\infty}  \right\},  $$
and the comparison perturbed hybrid system
\begin{align*}
\hspace{-0.2cm} 
\mathcal H_w : & \left\{
\begin{array}{rcll} 
\hspace{-0.2cm}
\begin{bmatrix} \,\dot w\, \\ \dot p \\ \dot q \end{bmatrix} &\hspace{-1ex} = &\hspace{-1ex} 
\begin{bmatrix} 
f'(w,\bar{\nu}) \\ 1 \\ 0 
\end{bmatrix} & x \in C_w
\\[17pt]
\hspace{-0.2cm}
\begin{bmatrix} 
w^+  \\ p^+ \\ q^+ 
\end{bmatrix} &\hspace{-1ex} = &\hspace{-1ex} 
\begin{bmatrix} 
g'(w,\bar{\nu}) \\ p \\ q + 1 \end{bmatrix} & x \in D_w,
\end{array} \right.
\end{align*}
 with 
\begin{align*}
f'(w,\bar \nu)  := &  - \displaystyle\frac{1}{2 p_M} w +  4 p_M  \bar\nu^2   \qquad w \in C_w   
\\[3pt]  
g'(w,\bar \nu) := &  \Big[1  - \displaystyle\frac{1}{2 p_M}\Big]  w  
+ c \bar\nu^2  \qquad w \in D_w,
\end{align*}
where $C_w := \mathbb R_{\geq 0} \times \dom \phi$ and  
$D_w := \{ (w,p,q) \in \mathbb R_{\geq 0} \times \dom \phi \,:\, (p,q+1) \in \dom \phi \}$.
Thus, we conclude using Lemma \ref{lem:comparison} that the solutions $\phi_w$ to $\mathcal H_{w}$ and $\phi$ to  $\mathcal H$ in \eqref{597} obtained from \eqref{606} satisfy $W(\phi(t,j))  \leq \phi_w(t,j)$ for all  $(t,j)\in \dom \phi$, solving for  $\mathcal H_{w}$, it follows that there exist $a$, $b>0$ such that 
$$  W(\phi(t,j)) \leq  W(\phi(0,0)) e^{- a(t+j)}  + b \bar\nu, $$ 
and ISS of $\mathcal H_\nu'$ follows.
\end{proof} 

\subsection{UES and ISS Under HPE} 

The following is a relaxed PE property, which captures the richness of signals that may fail to be PE if considered as functions of purely continuous or purely discrete time.

\begin{definition} [HPE] \label{def:hpe}
The pair $(A,B)$ of hybrid arcs $A, B: \dom A \rightarrow \mathbb{R}^{m_\zeta \times m_\zeta}$, {\it i.e.},  with $\dom A = \dom B$,  is said to be HPE if there exist $K$ and $\mu > 0$ such that
\begin{align} \label{eq:HPE}
\int_{E^A_{t_o,j_o,K}}  \hspace{-0.6cm}  \Phi_{AB}(s,i) d(s,i) \geq \mu I_{m_\zeta}
\qquad \forall (t_o,j_o) \in \dom A, 
\end{align}
where  $\Phi_{AB} : \dom A \rightarrow \mathbb{R}^{m_\zeta \times m_\zeta}$ is given by 
\begin{align*}
\Phi_{AB}(t,j) := 
\left\{ 
\begin{array}{cl} 
A(t,j) & \text{if} ~~  t \in \text{int}(I^j_A) 
\\ 
B(t,j) & \text{otherwise}.
\end{array}
\right.
\end{align*}
\ \vskip -10pt
\end{definition}

As for purely continuous-time systems, an important property of HPE is that it implies HUO. Theorem \ref{thm:PE=>UO}, below, generalizes to the realm of hybrid systems, the well-known fact that PE implies UO---see \cite{CLAN8,IOASUN}. Yet, Theorem \ref{thm:PE=>UO} is not a direct extension since its proof approach is original. For instance, it differs from that used in \cite{anderson1969new} for continuous-time systems by being direct and not relying on many intermediate results. 

\begin{assumption}[Structural Properties] \label{ass:str}
 For each $(t,j) \in \dom A$,   $ A(t,j) = A(t,j)^\top \geq 0$,   $B(t,j) = B(t,j)^\top \geq 0$,  and $|B(t,j)|_{\infty} \leq 1$.
\end{assumption}
\begin{theorem}[HPE implies HUO]  \label{thm:PE=>UO}
Consider the hybrid system $\mathcal H'_0$ under Assumptions \ref{ass:linf}  and \ref{ass:str},  and let the  pair $\left( A,B \right)$ be  HPE. Then,  the  pair $\left( A,B \right)$ is HUO. 
\end{theorem}
\begin{proof}
Under Assumption \ref{ass:str}, it follows that Assumption \ref{ass:UGS} holds with $P = I_{m_\theta}$,  $Q_c(t,j) = A(t,j)$,  and $Q_d(t,j) = B(t,j)$. Therefore, to verify the HUO property, it suffices to find $\mu_o > 0$ such that,  for each $(t_o,j_o) \in \dom A$,  we have  
\begin{equation}
  \label{eqtoshow}
\int_{E^A_{t_o,j_o,K}}  \hspace{-0.8cm} \Gamma_o(s,j) 
d(s,j)  
 \geq \mu_o I_{m_\zeta}, 
\end{equation}
where we defined 
$$
\Gamma_o(s,j):=\mathcal{M} \left((s,j), (t_o,j_o) \right)^\top  \Phi_{AB}(s,j) \mathcal{M}\left((s,j), (t_o,j_o) \right)
$$
to compact the notation,   $K$ comes from the HPE of $(A,B)$.
Then, to establish  \eqref{eqtoshow},  we show that,  for each $\zeta_o \in \mathbb{R}^{m_\zeta}$,  
$$
\zeta_o^\top  \int_{E^A_{t_o,j_o,K}} \hspace{-0.5cm}
\Gamma_o(s,j) d(s,j)\,  \zeta_o
\geq \mu_o |\zeta_o|^2. $$
To that end, first we note that 
\begin{equation*}
\begin{aligned} 
\zeta_o^\top  \int_{E^A_{t_o,j_o,K}}  \hspace{-0.9cm}
\Gamma_o(s,j) d(s,j)\,  \zeta_o
=     \int_{E^A_{t_o,j_o,K}}  \hspace{-0.9cm} \zeta(s,j)^\top   \Phi_{AB}(s,j) \zeta(s,j) d(s,j)
\end{aligned}
\end{equation*}
and  we proceed to find $\mu_o> 0$ such that 
\begin{equation}
  \label{eqUES}  
 \tilde V  := \int_{E^A_{t_o,j_o,K}}  \hspace{-0.6cm} \zeta(s,j)^\top   \Phi_{AB}(s,j) \zeta(s,j) d(s,j) 
  \geq   \mu_o  |\zeta_o|^2.
\end{equation}
So, to prove \eqref{eqUES},  we express $\tilde{V}$ as 
 \begin{equation}  \label{eqDV}    
 \begin{aligned}
\tilde{V} = \sum\limits_{j=j_o}^{m_K} 
 V_F (j) + \sum\limits_{j=j_o}^{m_K}  V_G (j),
 \end{aligned}
 \end{equation}
 where 
 \begin{align}
 V_F (j) & :=  \int^{t_j + 1}_{t_j} \hspace{-0.3cm}  \zeta(s,j)^\top  A(s,j) 
 \zeta(s,j) ds,  \label{eqVF}
\\
 V_G (j) & :=   \zeta(t_{j+1},j)^\top  B(t_{j+1},j) 
 \zeta(t_{j+1},j),  \label{eqVG}
 \end{align} 
and we compute suitable lower bounds for these functions. 
\\
In regards to $V_F$, we show that, for each $\rho>0$  and for each $j \in \{j_o,\ldots,m_K\}$, 
 \begin{equation} \label{eqVFcond}
 \begin{aligned}
 & V_F (j)  \geq   \frac{\rho}{1+\rho}\int_{t_{j}}^{t_{j+1}} \Big| A(s,j)^{\frac{1}{2}} \zeta_o \Big|^2 ds \\ & 
 -  \rho  (\bar{A}^2 + 2 \bar{A})(2(j-j_o)+1)(t_{j+1}-t_j)(t_{j+1}-t_{j_o}+1) \tilde{V}. 
  \end{aligned}
  \end{equation}
To better see this, we note that
 \begin{equation} \label{eqSol}
 \begin{aligned}
\hspace{-0.4cm} \zeta(t,&\, j) = \zeta_o - \int_{t_j}^t A(s,j) \zeta(s,j) ds    \\ 
\hspace{-0.4cm} &- \sum\limits_{i=j_0}^{j-1} \int_{t_{i}}^{t_{i+1}}  \hspace{-0.4cm} A(s,i) \zeta(s,i) ds 
- \sum\limits_{i=j_0}^{j-1} B(t_{i+1},i) \zeta(t_{i+1},i).
 \end{aligned}
 \end{equation}
Hence,  we obtain 
 \begin{align*}
 V_F (j) =   \int_{t_{j}}^{t_{j+1}} & 
 \Big|\, A(s,j)^{\frac{1}{2}} \zeta_o   + A(s,j)^{\frac{1}{2}}\! \int_{t_{j}}^{s} A(u,j) \zeta(u,j) du   \\ &  
+  A(s,j)^{\frac{1}{2}} \sum\limits_{k=j_o}^{j-1}\int_{t_{k}}^{t_{k+1}}A(u,k) \zeta(u,k)du  
\\ &   
   + A(s,j)^{\frac{1}{2}} \sum\limits_{k=j_o}^{j-1}B(t_{k+1},k) \zeta(t_{k+1},k)  \,\Big|^2 ds.
\end{align*}
 Next, using the fact that $ |a-b|^2 \geq \frac{\rho}{1+\rho}|a|^2-\rho |b|^2$ for all $\rho >0, $ we obtain
 \begin{align*}
 V_F(j)  \geq  -  \rho\int_{t_{j}}^{t_{j+1}} &\Big|\, A(s,j)^{\frac{1}{2}} \int_{t_{j}}^s A(u,j) \zeta(u,j) du  
  \\ &   + A(s,j)^{\frac{1}{2}} \sum\limits_{k=j_o}^{j-1} \int_{t_{k}}^{t_{k+1}} A(u,k) \zeta(u,k) du 
 \\ &   + A(s,j)^{\frac{1}{2}} \sum\limits_{k=j_o}^{j-1} B(t_{k+1},k) \zeta(t_{k+1},k) \Big|^2 ds \\ & 
 + \frac{\rho}{1+\rho}  \int_{t_{j}}^{t_{j+1}} \Big| A(s,j)^{\frac{1}{2}} \zeta_o \Big|^2 ds.
 \end{align*}

 Furthermore, using the boundedness of $A$ according to Assumption \ref{ass:linf} and the fact that 
 $ |\sum\limits_{i=1}^N a_i|^2 \leq N  \sum\limits_{i=1}^N|a_i|^2, $ we obtain
 
  \begin{align*}
V_F (j)  \geq 
-  \rho \bar{A}(2( & j-j_o)+1)  \times  \int_{t_{j}}^{t_{j+1}} 
 \left[ \bigg|\int_{t_{j}}^s A(u,j) \zeta(u,j) du \bigg|^2 
 \right. \\ &  \left. + \sum\limits_{k=j_o}^{j-1} \left| \int_{t_{k}}^{t_{k+1}}A(u,k) \zeta(u,k)du \right|^2  \right. \\ &  \left. + \sum\limits_{k=j_o}^{j-1} \left|B(t_{k+1},k) \zeta(t_{k+1},k) \right|^2 \right]ds 
 \\ & + \frac{\rho}{1+\rho}\int_{t_{j}}^{t_{j+1}} \left| A(s,j)^{\frac{1}{2}} \zeta_o \right|^2 ds.
 \end{align*}
 Next, using the triangular and Cauchy-Schwartz inequalities, we obtain\\[-15pt] 
 \begin{align*}
 V_F (j)  \geq  -  \rho & \bar{A}(2(  j-j_o)+1)  \\ & 
 \times  \int_{t_{j}}^{t_{j+1}} 
\bigg[ (s-t_j) \int_{t_{j}}^s \left| A(u,j) \zeta(u,j)\right|^2 du  \\ &  
 +  \sum\limits_{k=j_o}^{j-1} (t_{k+1}-t_k) \int_{t_{k}}^{t_{k+1}}  \left| A(u,k) \zeta(u,k) \right|^2 du  
 \\[-3.5pt] & 
+  \sum\limits_{k=j_o}^{j-1}\left|B(t_{k+1},k) \zeta(t_{k+1},k) \right|^2 \bigg] ds  \\ & + \frac{\rho}{1+\rho} \int_{t_{j}}^{t_{j+1}} \left| A(s,j)^{\frac{1}{2}} \zeta_o \right|^2 ds.
 \end{align*}
 Then, using the fact that  
 \begin{align*}
 0 & \leq   (s-t_j)\int_{t_{j}}^s \left| A(u,j) \zeta(u,j)\right|^2 du   
\\ &  \leq (t_{j+1}-t_j)\int_{t_{j}}^{t_{j+1}} \left| A(u,j) \zeta(u,j)\right|^2 du,  
 \end{align*}
 we obtain
 \begin{align*}
  V_F (j)  \geq   -  \rho \bar{A}(2( & j-j_o)+1)(t_{j+1}-t_j)(t_{j+1}-t_{j_o}+1) \\  
 \times & \left[\sum\limits_{k=j_o}^{j} \int_{t_{k}}^{t_{k+1}}  \left| A(s,k) \zeta(s,k) \right|^2 ds 
 \right. \\  & \qquad \left. 
 + \sum  \limits_{k=j_o}^{j-1} \left|B(t_{k+1},k) \zeta(t_{k+1},k) \right|^2 \right] 
 \\ & + \frac{\rho}{1+\rho}\int_{t_{j}}^{t_{j+1}} \left| A(s,j)^{\frac{1}{2}} \zeta_o \right|^2 ds,
 \end{align*}
and using  $A  = A^{\frac{1}{2}} A^{\frac{1}{2}}$    and  $B  = B^{\frac{1}{2}} B^{\frac{1}{2}}$,  we conclude that 
 \begin{align*}
 V_F (j)   \geq &  -  \rho \bar{A}( \bar{A}+2)(2(j-j_o)+1)(t_{j+1}-t_j) 
 \\ & \times (t_{j+1}-t_{j_o}+1) 
 \left[\sum\limits_{k=j_o}^{j} 
 \int_{t_{k}}^{t_{k+1}}  \left| A^{\frac{1}{2}}(s,k) \zeta(s,k) \right|^2 \!ds 
 \right. \\ & \left. + \frac{1}{2} \sum\limits_{k=j_o}^{j-1}\left|B(t_{k+1},k)^{\frac{1}{2}} \zeta(t_{k+1},k) \right|^2 \right] \\ &
 + \frac{\rho}{1+\rho}\int_{t_{j}}^{t_{j+1}} \left| A(s,j)^{\frac{1}{2}} \zeta_o \right|^2 ds   .
 \end{align*}
Hence, \eqref{eqVFcond} follows. Next, we show that, for each $\rho > 0$  and for each $j \in \{j_o,\ldots,m_K\}$,  the following inequality holds:
\begin{align}\nonumber
 V_G (j)  \, \geq\, \frac{-\rho}{2} (2(j & -j_o)+1) (\bar{A}(t_{j+1}-t_{j_o})+2) \tilde{V} \\ &
 + \frac{\rho/2}{1+\rho}\left| B(t_{j+1},j)^\frac{1}{2} \zeta_o \right|^2, 
  \label{eqVGcond}
 \end{align}
For this, we first note that
$ 2 V_G (j)  =   \left|B^{\frac{1}{2}}(t_{j+1},j) \zeta(t_{j+1},j)\right|^2. $
Then, using \eqref{eqSol}, we obtain 
 \begin{align*}
  V_G(j)  =  \frac{1}{2}\,
  \bigg|\, & B(t_{j+1},j)^{\frac{1}{2}} \zeta_o  \\ & 
  + B(t_{j+1},j)^{\frac{1}{2}} \sum\limits_{k=j_o}^{j}\int_{t_{k}}^{t_{k+1}}\hspace{-3pt} A(u,k)  \zeta(u,k)du  \\ &   + B(t_{j+1},j)^{\frac{1}{2}} \sum\limits_{k=j_o}^{j-1}B(t_{k+1},k) \zeta(t_{k+1},k) \,\bigg|^2.
   \end{align*}
Now, using the fact that  $ |a-b|^2 \geq \frac{\rho}{1+\rho}|a|^2-\rho |b|^2$ for all $\rho >0$,  we obtain 
 \begin{align*} 
 V_G (j)   \geq  - \frac{\rho}{2} \bigg| B & (t_{j+1},j)^{\frac{1}{2}} \sum\limits_{k=j_o}^{j} \int_{t_{k}}^{t_{k+1}}
  \hspace{-3pt} A(u,k) \zeta(u,k)du   \\ &   
  + B^{\frac{1}{2}}(t_{j+1},j) \sum\limits_{k=j_o}^{j-1}B(t_{k+1},k) \zeta(t_{k+1},k) \bigg|^2  \\ & 
  +  \frac{\rho/2}{1+\rho} \Big| B(t_{j+1},j)^\frac{1}{2} \zeta_o \Big|^2.
 \end{align*}
\\
 Next, using $ |\sum\limits_{i=1}^N a_i|^2 \leq N  \sum\limits_{i=1}^N|a_i|^2, $  the boundedness of $B$ by $1$, and the Cauchy-Schwartz inequality, we obtain
 \begin{align*}
& V_G (j) \geq  \frac{-\rho}{2}(2(j-j_o)+1)   \\ &
 \times \left[ \sum\limits_{k=j_o}^{j}  \int_{t_{k}}^{t_{k+1}} \hspace{-0.4cm} (t_{k+1}-t_k)  \left| A(u,k) \zeta(u,k)\right|^2du 
 \right. \\ & \left. 
 + \sum\limits_{k=j_o}^{j-1} \left|B^{\frac{1}{2}}(t_{j+1},j) \zeta(t_{k+1},k)\right|^2\right]  + \frac{\rho/2}{1+\rho}\left| B(t_{j+1},j)^\frac{1}{2} \zeta_o \right|^2.
 \end{align*}
 Finally, using the boundedness of $A$, according to Assumption~\ref{ass:linf}, we obtain   
  \begin{align*}
 V_G (j)  \geq  & \, \frac{-\rho}{2} (2(j-j_o)+1)(\bar{A}(t_{j+1}-t_{j_o})+2)\qquad\qquad \\ & \qquad
 \times  \left[ \sum\limits_{k=j_o}^{j}   \int_{t_{k}}^{t_{k+1}}  \left|A(u,k)^{\frac{1}{2}} \zeta(u,k)\right|^2 du   \right. \\ & \quad \left. \qquad\quad
 + \sum\limits_{k=j_o}^{j-1}  \frac{1}{2}\left|B(t_{j+1},j)^{\frac{1}{2}}
 \zeta(t_{j+1},j)\right|^2\right] \\ &
\quad  + \frac{\rho/2}{1+\rho}\left| B(t_{j+1},j)^\frac{1}{2} \zeta_o \right|^2.
 \end{align*}
Hence,  \eqref{eqVGcond} follows. 
Now, combining \eqref{eqVFcond} and \eqref{eqVGcond}, we obtain the following upper bound on $\tilde{V}$ for each $\rho>0$: 
 \begin{align*}
&  \tilde{V} \geq  \frac{\rho}{1+\rho}  \sum\limits_{j=j_o}^{m_K-1} \left| B(t_{j+1},j)^\frac{1}{2} \zeta_o \right|^2 \qquad
 \\ & 
 + \frac{2\rho}{1+\rho} \sum\limits_{j=j_o}^{m_K} \int_{t_{j}}^{t_{j+1}} \left| A(s,j)^{\frac{1}{2}} \zeta_o \right|^2 ds 
 \\ & 
 - \frac{\rho}{2} \tilde{V}   \sum\limits_{j=j_o}^{m_K} (2(j-j_o)+1)(\bar{A}(t_{j+1}-t_{j_o})+2) 
 \\ & 
 - \rho \bar{A} (\bar{A}+2) \tilde{V} \hspace{-0.1cm} \sum\limits_{j=j_o}^{m_K}  \hspace{-0.1cm} (2(j-j_o)+1)(t_{j+1}-t_j)(t_{j+1}-t_{j_o}+1). 
 \end{align*}
 Hence,
 \begin{align*}
& \tilde{V}   \geq \frac{2\rho |\zeta_o|^2}{1+\rho}   
\int_{E^A_{t_o,j_o,K}}  \hspace{-0.7cm}  \Phi_{AB}(s,i) d(s,i) -\rho(m_K-j_o+1)^2  \tilde{V} 
 \\ & \hspace{-0.6cm} \times   \left[\frac{\bar{A}}{2}(s_K - t_{j_o})+2) + \bar{A} (\bar{A}+1)(s_K - t_{j_o})(s_K - t_{j_o}+1) \right] 
 \end{align*}
\\
 Finally, using the HPE of the pair $(A,B)$,  we conclude that
 \begin{align}\nonumber
 \tilde{V}  \geq  \frac{2\rho \mu}{1+\rho}&  |\zeta_o|^2
- \rho(K+2)^2 \\ & \hspace{-4em} \times    \left[\frac{\bar{A}}{2}(K+1)+1+\bar{A}(\bar{A}+2)(K+1)(K+2)\right]\tilde{V}.
\label{toto} 
\end{align}
Thus, \eqref{eqUES} follows by choosing 
 $$ \rho := \frac{1/(K+2)^2}{\frac{\bar{A}}{2}(K+1)+1+\bar{A}(\bar{A}+2)(K+1)(K+2)}. $$
 \vskip -13pt
\end{proof}

The importance of Theorem \ref{thm:PE=>UO} relies on the following statement,  whose proof is direct and, yet, it  generalizes similar results available for continuous- or discrete-time systems.

\begin{theorem}[UES + ISS under HPE]  \label{thm:PE=>UES}
Consider the hybrid system $\mathcal H'_0$, defined by \eqref{606} with $\nu\equiv 0$, under Assumptions \ref{ass:linf} and \ref{ass:str}, and let the pair $\left(A,B \right)$ be HPE.   Then, the origin $\{\zeta=0\}$  is UES for $\mathcal H'_0$ and $\mathcal H'_\nu$ is ISS with respect to $u$.  
\end{theorem}

\begin{proof}
After Theorem \ref{thm:PE=>UO},  the HUO property holds.  Under Assumption \ref{ass:str}, it follows that Assumption \ref{ass:UGS} holds with $P = I_{m_\theta}$,  $Q_c(t,j) = A(t,j)$,  and $Q_d(t,j) = B(t,j)$. Thus, the statement follows from a direct application of Theorem \ref{HUO=>UES}. 
\end{proof}

\section{Adaptive Estimation under HPE} \label{sec:egs} 

\subsection{The Hybrid gradient-descent algorithm}

To put our contributions in perspective,  
we first consider a classical identification problem,  based on the linear regression model 
\begin{equation} \label{190}
y = \psi^\top\theta, 
\end{equation}
where $\psi : \dom \psi \rightarrow \mathbb{R}^{m_\theta}$ is the regressor,  $\theta \in \mathbb{R}^{m_\theta}$ is a constant vector of unknown parameters, and $y : \dom y \rightarrow \mathbb{R}$ is the output. Usually, the domains of $y$ and $\psi$ are considered to be subsets of the real numbers or the natural numbers.  Then, an estimate of $\theta$, denoted $\hat \theta$, may be carried out dynamically, in function of the tracking error $ e := \hat{y} - y $, where $\hat y := \psi^\top\hat\theta$.   A well-known identification law is based on the minimization of the cost $J(e) := (1/2) e^2$ and defined by the gradient of the latter. 
 
In the continuous-time setting, {\it i.e.}, if the regressor's domain is $\dom \psi = [0, + \infty)$, the gradient-based update law for $\hat \theta$ is given by $\dot{\hat{\theta}}=-\gamma\nabla_{\hat{\theta}}J(e)$, where $\nabla_{\hat{\theta}} J$ denotes the gradient of $J = (1/2) (\psi^\top\hat \theta - \psi^\top\theta)^2$ with respect to $\hat{\theta}$. Hence,
 \begin{equation} \label{eqn:algo1}
\dot{\hat\theta}= - \gamma\psi(t)[\psi(t)^{\top}\hat{\theta}-y(t)],\quad \gamma>0
 \end{equation}
---see  \cite{narendra2012stable}.  In this case,  the dynamics of the estimation error $\tilde\theta:= \hat\theta - \theta$ is given by 
\begin{equation} \label{eqn:algo1bis}
\dot{\tilde\theta} = -\gamma \psi(t) \psi(t)^\top \tilde\theta
\end{equation}
and it is well-known (see, {\it e.g.}, \cite{Narendra}) that, if $\psi$ is bounded, the following condition of continuous-time PE (CPE) is necessary and sufficient for UES of the origin for \eqref{eqn:algo1bis}. 
\begin{enumerate}[topsep=0pt,label={(CPE)},leftmargin=*]
\item \label{item:c2}  There exist $T > 0$ and $\mu > 0$ such that
\begin{equation}
  \label{922}  \int_{t}^{t + T} \hspace{-0.2cm} \psi(s) \psi(s)^{\top} ds \geq \mu I_{m_\theta} \qquad \forall t \geq 0.
\end{equation}
 \end{enumerate}
Moreover, a lower bound on the convergence rate is provided in \cite{brockett2000rate}, \cite{AND77}, and \cite{al:LTVsystSCL},  and a strict Lyapunov function is constructed in \cite{JP4-SRIKANT-IJC}.  

In the discrete-time setting,  {\it i.e.},  if the regressor's domain is $\dom \psi = \mathbb Z_{\geq 0} $, the gradient algorithm is given by
 \begin{equation} \label{eqn:algo1b}
     \hat{\theta}(t+1)=\hat{\theta}(t)-\sigma(t) \nabla_{\hat{\theta}} J(e),
\end{equation}
 where $\sigma : \mathbb Z_{\geq 0} \rightarrow [0,1]$ is given by $\sigma(t) := \frac{\gamma}{1 + \gamma |\psi(t)|^{2}}$, and $\gamma > 0$ is the adaptation rate \cite{tao2003adaptive}. Therefore, the dynamics of the estimation error is given by 
\begin{equation} \label{eqn:algo1bisd}
\tilde\theta^{+} = \left( I_{m_\theta} - \frac{\gamma \psi(t) \psi(t)^\top}{1 + \gamma |\psi(t)|^{2}} \right) \tilde\theta.
\end{equation}
 In the latter case, the discrete-time PE condition reads---cf. \cite{bai-sastry_DT-PE,ASTBOH}:
\begin{enumerate}[topsep=0pt,label={(DPE)},leftmargin=*]
 \item \label{item:c3} There exist $N > 0$ and $\mu >0$ such that
\begin{equation}
  \label{939}  \sum_{s = 0}^{N} \psi(s) \psi(s)^{\top} \geq \mu I_{m_\theta}.  
\end{equation}
\end{enumerate}

\begin{remark}
   Note that some of the existing approaches to analyze \eqref{eqn:algo1bis} translate naturally to the analysis of \eqref{eqn:algo1bisd} under DPE; see, {\it e.g.}, \cite{tao2003adaptive}.   Other relaxed forms of PE are also available,  but these do not lead to uniform forms of convergence---see {\it e.g.}, \cite{MORNAR,praly2017convergence,chowdhary_IJC2014}.  
\end{remark}

Even though PE as defined above,  in continuous {\it or} discrete time,  is necessary for UES,   in some simple cases it may be over-restrictive.  For instance,  when the data $(\psi,y)$ of the linear regression model \eqref{190} is hybrid; namely,  when it is allowed to exhibit both continuous- and  discrete-time evolution,  to have  
\begin{align} \label{h-i-omodel}
y(t,j) = \psi(t,j)^\top \theta \qquad (t,j) \in \dom \psi. 
\end{align}
In this case,  the classical gradient-descent algorithms recalled above are ineffective.  This is because the continuous-time update law \eqref{eqn:algo1}  exploits the data only on the time intervals on which they evolve continuously,  while the discrete-time gradient algorithm \eqref{eqn:algo1b} exploits the data only at discrete time instants. If, in contrast to this, the regressor is hybrid,   we design a  hybrid gradient-descent algorithm in a way that whenever the data 
$(\psi,y)$ {\it jump}, {\it i.e.}, undergo an instantaneous change,  $\hat{\theta}$ is updated via \eqref{eqn:algo1b}; whenever the data $(\psi,y)$ {\it flow}, {\it i.e.}, evolve continuously, $\hat{\theta}$ is updated via \eqref{eqn:algo1}. More precisely, 
\begin{enumerate}[label={(HG\arabic*)},leftmargin=*]
\item \label{item:HG2} when $\psi$ flows, that is, for all $t\in \mbox{int}(I^j_\psi)$, with $I^j_\psi := \{t : (t,j) \in \dom \psi\}$,  $\hat{\theta}$ is updated  by
\begin{equation*}
\dot{\hat{\theta}} = -\gamma\psi(t,j) [\,\psi(t,j)^{\top}\hat{\theta}(t,j)-y(t,j)\,].
\end{equation*}
\item \label{item:HG1} 
Alternatively, when $\psi$ jumps, that is, for all $(t,j) \in \dom \psi$ such that $(t,j+1) \in \dom \psi$,  the estimate $\hat{\theta}$ is updated using
\begin{equation*}
\hspace{-1.2cm} \hat{\theta}(t,j+1)=\hat{\theta}(t,j)-\frac{\gamma \psi(t,j)[\,\psi(t,j)^{\top}\hat{\theta}(t,j)-y(t,j)\,]}{1 + \gamma |\psi(t,j)|^{2}}.
\end{equation*}
\end{enumerate} 
 Then,  the dynamics of the parameter estimation error $\tilde\theta = \hat{\theta} - \theta$ is governed by the hybrid system $\mathcal H'_0$,  \eqref{606}, with $\zeta := \tilde\theta$, $\nu\equiv 0$, 
 \begin{align}
\label{1419}   A(t,j) &:= \gamma \psi(t,j) \psi(t,j)^\top, \\
\label{1420}   B(t,j) &:= - \displaystyle\frac{\gamma \psi(t,j) \psi(t,j)^{\!\top}\!\!\!}{1 + \gamma |\psi(t,j)|^{2}},  
 \end{align}

\noindent which  satisfy the structural properties in Assumption \ref{ass:str}. Furthermore, it is assumed that, by design, there exists $\bar{\psi}>0$ such that $|\psi|_\infty \leq \bar{\psi}$  holds and the pair $(A,B)$ is HPE. 
  
In the following example,  we illustrate a scenario,  where the regressor $\psi$ in  \eqref{h-i-omodel} is a hybrid signal.

\begin{example} [Regressor gathering  real-time and old data] \label{exp1}
Consider the continuous-time input-output model 
\begin{equation} \label{190+}
y_1(t) = \psi_1(t)^\top\theta  \qquad t \geq 0, 
\end{equation}
where $\psi_1 : \mathbb{R}_{\geq  0} \rightarrow \mathbb{R}^{m_\theta}$ is the input and $y_1 :  \mathbb{R}_{\geq  0} \rightarrow \mathbb{R}$ is the output.   The pair $(\psi_1,y_1)$  defines the  \textit{real-time} input-output data.  On the other hand,  we assume that we have a memory containing a pair of \textit{old} input-output data,  which we denote by $(\psi_2,y_2)$.   The old data needs to be treated at specific times defining the sequence  $\{t_1, t_2,... , t_J \} \subset \mathbb{R}_{\geq 0}$  with  $t_j \leq t_{j+1}$. 
As a result,  the old input-output data satisfy
\begin{equation} \label{191}
y_2(t_j) = \psi_2(t_j)^\top\theta  \qquad \forall 
j \in \{1,2,... ,J\}.  
\end{equation}

The incorporation of old data can be done periodically,  it can also be triggered by an external supervisory algorithm.  
As a result,  we introduce  the hybrid time domain 
$$  \dom \psi := [0,t_1] \times \{ 0 \} \cup [t_1, t_2] \times \{ 1 \}  \cup ...  \cup [t_J, +\infty) \times \{J \}.  $$
Furthermore,  we introduce the pair of hybrid 
input-output data,  gathering both old and 
real-time data, given by 
$$ 
\psi(t,j)  :=  \left\{ 
\begin{matrix} 
\psi_1(t) &  \text{if} ~  t \in \text{int}(I^j_\psi) 
\\
\psi_2(t_{j+1}) & \text{otherwise},
\end{matrix}
\right.   $$
and
$$ 
y(t,j)  := \left\{ 
\begin{matrix} 
y_1(t) & \text{if} ~ t \in \text{int} (I^j_\psi) 
\\
y_2(t_{j+1}) &  \text{otherwise}.
\end{matrix}  
\right.  
$$
The pair of hybrid input-output data is related to the parameter 
$\theta$ according to \eqref{h-i-omodel}. 

The hybrid gradient algorithm in this context allows to continuously explore real-time data on the open intervals $\text{int}(I^j_\psi)$,  $j \in \{1,2,...,J\}$,  and to discretely exploit old data over the sequence of times  $\{t_j \}^\infty_{j=1}$.  
\end{example}

\begin{remark}
When the hybrid arc $\psi$ is eventually continuous (respectively, eventually discrete or Zeno),  HPE of the pair $(A,B)$ reduces to CPE of $\psi \psi^\top$ (respectively,  DPE).  Furthermore, when the regressor $\psi$ is scalar (i.e, $m_\theta = 1$),  HPE of the pair $(A,B)$ implies that either CPE or DPE holds.  However, in the general case  that $m_\theta > 1$, it is possible that HPE hold,  but none of the conditions  CPE and DPE  be  satisfied.  
\end{remark}

 For illustration, let us consider the linear relationship in \eqref{h-i-omodel},  where the regressor function $\psi$ is given by
\begin{align}  \label{261}
\hspace{-0.2cm} \psi(t) := 
\left\{  \begin{matrix}
 [\sin(t) \quad 0]^{\top} & \hspace{-0.2cm}
 \forall t \in (2j\pi,2(j+1)\pi), ~ j \in \mathbb Z_{\geq 0} 
 \\[3pt]
[0.5 \quad 1]^{\top} &   \text{otherwise},
\end{matrix}
\right.
\end{align}
so,  over successive continuous intervals of time,  
\begin{equation*}  
\psi(t) \psi(t)^{\top} = 
\begin{bmatrix}  \sin(t)^2 & 0 \\[2pt] 0 & 0 \end{bmatrix}
 \forall t \in (2j\pi,2(j+1)\pi),  ~  j \in \mathbb Z_{\geq 0}, 
\end{equation*}
while at discrete instants,
\[ 
\psi(t) \psi(t)^{\top} := \begin{bmatrix} 0.25 & 0.5 \\ 0.5 & 1 \end{bmatrix}  \qquad \forall t \in \{2\pi, 4\pi, ... \}.  
\]
The function $\psi \psi^\top$ defined in \eqref{261} does not satisfy neither CPE nor DPE. However, the corresponding maps $A$ and $B$ (with  $\gamma=1$), defined on 
$\dom A = \dom B = \bigcup^\infty_{j = 0} [2j\pi, 2(j+1)\pi], $
are given by 
\begin{align*}  
A(t,j)  & =
\begin{array}{cl}
 \begin{bmatrix} 
\sin(t)^2 & 0 \\ 0 & 0
\end{bmatrix},
\end{array}
B(t,j)  =  
\begin{array}{cl}
\hspace{-1ex} 
\begin{bmatrix} 
0.1111 & 0.2222 
\\ 
0.2222 & 0.4444 
\end{bmatrix}.  \\
\end{array}
\end{align*}
The pair $(A,B)$ is HPE,  with  $K=2\pi+1$ and $\mu=0.21$.

 HPE is less conservative than its counterparts CPE and DPE as it captures  the fact that the richness of a signal may be enhanced by an appropriate mingling of exciting flows and jumps, which, otherwise, are insufficient to guarantee that neither \eqref{922} nor \eqref{939} hold.  The latter being necessary,  $\tilde\theta\not\to 0$ in either case, but  $\tilde\theta\to 0$ under the hybrid gradient-descent algorithm---see Fig. \ref{fig:hybrid} below.
\begin{figure}[h!]
	\begin{center}
          \includegraphics[width=87mm]{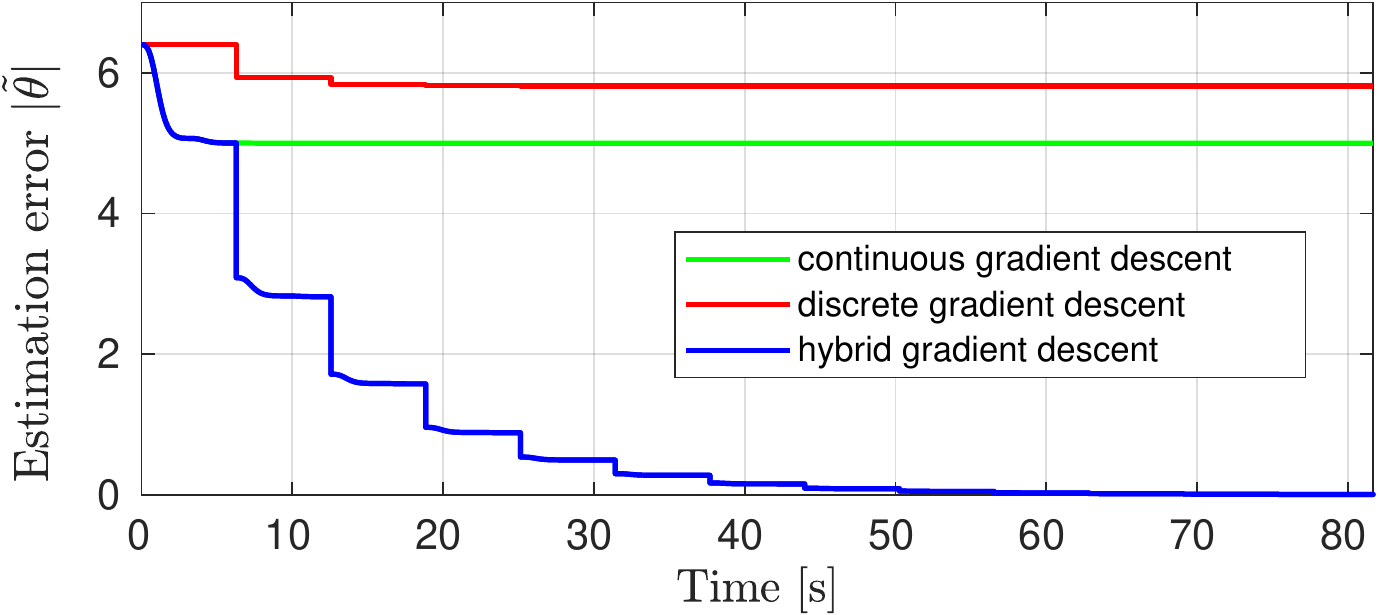} \\
	\end{center}
	\caption{Evolution of the norm of the parameter error $\tilde\theta$ using continuous, discrete,  and hybrid gradient algorithms}
	\label{fig:hybrid}	
\end{figure}
 The previous observations are captured in the following statement that follows as a direct corollary of Theorem \ref{thm:PE=>UES} and yet, covers the gradient-based algorithms for purely continuous- or discrete-time systems,  and stems as a corollary of Theorem \ref{thm:PE=>UES}. 
\begin{corollary} \label{cor1}
Consider the hybrid system $\mathcal H'_0$, defined by \eqref{608} with $A$ and $B$ as in \eqref{1419} and \eqref{1420},  respectively.  Let $\nu\equiv 0$ and assume that the pair $(A,B)$ is HPE and,  there exists $\bar{\psi}>0$ such that $|\psi|_\infty \leq \bar{\psi}$.  
Then,  the origin $\{\tilde\theta = 0\}$  is UES and the system's trajectories satisfy \eqref{608} with computable parameters 
$\kappa(\bar{\psi},\mu,K)$ and $\lambda(\bar{\psi},\mu,K)$.
\end{corollary}

\subsection{Hybrid adaptive Observer/Identifier design}

We address now a classical problem of adaptive observer/identifier design, recast in the realm of hybrid systems.  We show how well-known designs may be applied using excitation signals that  flow and jump.  For clarity, we start by revisiting an existing  design-method  for continuous- and discrete-time systems. 

\subsubsection{Rationale}

Consider the problem of estimating the state $x \in \mathbb R^{m_x}$,  and a vector of unknown constant parameters $\theta \in \mathbb R^{m_\theta}$,  of a nonlinear system driven by an input $u \in \mathbb R^{m_u}$,  based on the measurement of an output $y := Hx$,  $y\in \mathbb R^{m_y}$. 
%
That is, to produce estimates $(\hat x,\hat \theta)$ such that the estimation errors $(e,  \tilde\theta) := (x- \hat x,\hat\theta-\theta)$ converge to zero asymptotically.

 If the output $y$ is measured continuously in time, indexed by the variable $t$,  the plant may be modeled by 
\begin{equation} \label{1304} 
\dot x = A_c(y,u) x  + \Psi_c(y,u) \theta,
\end{equation}
whereas,  if the output measurements are made at discrete instants,  indexed by $j$, we shall use the model
\begin{equation}
  \label{1304bis} 
   x^+ = A_d(y,u) x  + \Psi_d(y,u)  \theta.
\end{equation}
%

This choice of models is motivated by the abundant literature on estimator design for systems that are affine in the unmeasured variable $x$ and linear in $\theta$. The problem is completely solved for continuous- as well as for discrete-time systems. However, available designs may fail if the system's dynamics is hybrid. Say, if it is  governed by \eqref{1304} when $x$ belongs to a flow set $C \subset \mathbb{R}^{m_x} \times \mathbb{R}_{\geq 0} \times \mathbb{N}$, and by \eqref{1304bis},  when $x$ is in a jump set $D \subset  \mathbb{R}^{m_x} \times \mathbb{R}_{\geq 0} \times \mathbb{N}$.
More precisely, a more general model of the plant is the hybrid system
\begin{equation}
\begin{aligned} 
\Sigma_P : & \left\{ 
\begin{matrix}
\begin{matrix} 
\dot{x}\ =A_c(y,u) x + \Psi_c(y,u) \theta 
\end{matrix}
 & 
x  \in  C
\\[3pt]
\begin{matrix}
x^+ \!= A_d(y,u) x + \Psi_c(y,u) \theta  
\end{matrix}
&\   x  \in  D. 
\end{matrix} \right.
\end{aligned}
\label{SigmaP}
\end{equation}

Addressing the observer/identifier  design problem for hybrid systems is not purely motivated by intellectual curiosity (or challenge). We describe next a realistic scenario in which the problem arises naturally. 
\begin{remark} \label{remtd}
We stress that after obvious modifications, the developments that follow remain valid if the functions $A_c$, $A_d$, $\Psi_c$, and $\Psi_d$ as well as the sets $C$ and $D$ depend on the (hybrid) time. That is,  $\Sigma_P$ can have the general 
form
 \begin{equation*}
 \begin{aligned} 
 \Sigma_P : & \left\{ 
 \begin{matrix}
\begin{matrix} 
 \dot{x}=A_c(t,j,y,u) x + \Psi_c(t,j,y,u) \theta 
 \\ \dot{t} = 1,  \qquad \dot{j} = 0
 \end{matrix}
  & 
 (x,t,j)  \in  C
 \\[3pt]
 \begin{matrix}
 x^+= A_d(t,j,y,u) x + \Psi_c(t,j,y,u) \theta   \\
 t^+ = t, \qquad j^+  = j +1
 \end{matrix}
 &  (x,t,j) \in  D. 
 \end{matrix} \right.
 \end{aligned}
 \end{equation*}
\end{remark}

\begin{example} \label{example2}
Consider the continuous-time  plant \eqref{1304} and, for simplicity, let the entire state be measurable, that is, let $y = x$.  Furthermore, let 
$t \mapsto (y_1(t),u_1(t))$ be a pair of real-time input-output data defined in continuous time. 
Hence, by letting  $t \mapsto A_1(t) := A_c(y_1(t),u_1(t))$ and $t \mapsto B_1(t) :=  \Psi_c(y_1(t),u_1(t))$, we conclude that the real-time state variable $x_1$ is governed by 
\begin{equation} \label{1307-} 
\dot x_1 = A_1(t) x_1  + B_1(t) \theta  \qquad t \geq 0. 
\end{equation}

Now, suppose that the past experiences have generated other pairs of 
input-output data. Let $t \mapsto (y_2(t),u_2(t))$ be one such pairs, also defined in continuous time. We assume that we were able to save the input-output data corresponding to a specific sequence of times in the past
$\{\tau_1, \tau_2,... , \tau_J\} \subset \mathbb{R}_{\leq 0}$.  More specifically,  we are able to save the sequence $\{ (y_2(\tau_j), u_2(\tau_j)) \}^{J}_{j=1}$,  and a sequence of pairs 
$$  \{A_{2}(j), B_{2}(j) \}^J_{j=1} \subset \mathbb{R}^{m_x \times m_x} \times \mathbb{R}^{m_x \times m_\theta},  $$
 so that the old state vector $x_2$ satisfies  
\begin{equation} \label{1305} 
 x_2(\tau_{j+1}) = A_{2}(j) x_2(\tau_j)  + B_{2}(j) \theta,  \quad \forall j \in \{1,2,... , J\}.
\end{equation}
To compute the pair $j \mapsto (A_{2}(j), B_{2}(j))$,  we let  
$A_2 := A_c(y_2,u_2)$ and $B_2 :=  \Psi_c(y_2,u_2)$.
Furthermore,   by letting $\mathcal{M}_2$ be the transition matrix corresponding to the system $\dot{x} = A_2(t) x$,  we conclude that 
$$ x_2(\tau_{j+1}) = \mathcal{M}_2(\tau_{j+1}, \tau_j) x_2(\tau_j)  + \int^{\tau_{j+1}}_{\tau_j} \hspace{-0.3cm} \mathcal{M}_2(\tau_{j+1}, s) B_2(s) ds   \theta.   $$
Hence,  we obtain, for each $j \in \{1,2,...,J\}$, 
$ A_{2}(j) :=  \mathcal{M}_2(\tau_{j+1}, \tau_j)$ and  
$ B_{2}(j) := \int^{\tau_{j+1}}_{\tau_j} \hspace{-0.3cm} \mathcal{M}_2(\tau_{j+1}, s) B_2(s) ds.  $
As a result,  the old state $x_2$ is governed by the discrete system 
\begin{equation} \label{1305+} 
 x_2^+ = A_{2}(j) x_2  + B_{2}(j) \theta  \qquad j \in \{1,2,... , J\}.
\end{equation}
Now,  as in Example \ref{exp1},  we assume that the old data $(y_2,u_2)$ needs to be treated at specific times defining the sequence  $\{t_1, t_2, ... ,t_J \} \subset \mathbb{R}_{\geq 0}$  with  $t_j \leq t_{j+1}$. 
The latter takes us to introduce the hybrid domain 
$$   E   := [0, t_1] \times \{ 0 \} \cup [t_1,  t_2] \times \{ 1 \} \cup ...   \cup [t_J, +\infty) \times 
\{J\}.   $$

The dynamics of both old and real-time data is governed by a hybrid system in the form of $\Sigma_P$.   To see this,  we consider  the augmented state vector $(x,t,j) \in \mathbb{R}^{2m_x} \times E$, where $x := (x_1,x_2)$.  Furthermore,  we introduce the hybrid arcs defined on $E$ by
$
\bar \Psi_c(t,j)  :=  \begin{bmatrix} B_1(t) \\ 0  \end{bmatrix}, ~~
\bar \Psi_d(t,j)  := 
\begin{bmatrix} 0 \\ B_{2}(j)  \end{bmatrix},  $
\begin{align*}
\bar A_c(t,j) & :=  \text{blkdiag} \{A_1(t), 0\},
\\
\bar A_d(t,j) & :=  \text{blkdiag} 
\{ I_{m_x} , A_{2}(j) \}. 
\end{align*}
 Thus, the plant with the available 
 (old and real-time) data can be accurately modelled as in \eqref{SigmaP} with the right-hand sides and the flow and jump sets therein being dependent on the hybrid time $(t,j)$---cf. Remark \ref{remtd}. That is, we introduce the   time-varying hybrid system 
\begin{equation*}
\begin{aligned}
\bar \Sigma_P : 
 \left\{
\begin{matrix}
\hspace{-0.8cm} 
\begin{matrix}
\dot{x}  =  \bar A_c(t,j) x +   \bar \Psi_c (t,j) \theta
\end{matrix} \qquad  t\in \text{int}(I^j_E)
\\
\begin{matrix}
x^+  =   \bar A_d(t,j) x +  \bar \Psi_d (t,j) \theta 
\end{matrix}
\quad  (t,j), (t,j+1) \in E.
\end{matrix}
\right.\color{black} 
\end{aligned}
\end{equation*}
The scenario described above suggests that an efficient adaptive observer/identifier should be able to explore real-time data $(y_1,u_1)$ to estimate $\theta$ over each interval $I^j_E$. Moreover,  it should  exploit old data $(y_2,u_2)$ at the specific times $\{t_j\}^{J}_{j=1}$.  The latter sequence of times can be periodic or dictated an external supervisory algorithm. \end{example} 

In what follows of this section, we construct a dynamics hybrid observer/identifier and establish asymptotic stability of the origin, in the space of the estimation errors.  
The proposed observer/identifier design for $\Sigma_P$ builds upon the designs in \cite{989154} for continuous- and in \cite{guyader2003adaptive} for discrete-time systems,  \eqref{1304} and \eqref{1304bis}---we revisit them below.   As in \cite{loria2009adaptive},  we use a PE condition along the trajectories,  to guarantee the convergence of the estimation errors.   However,  in contrast to these and other similar works,  the design takes into account the fact that the system's solution is hybrid,  which flows and jumps.  As we show on a concrete example,  the richness induced by the hybrid behavior is fundamental to achieve the estimation goals, when other algorithms fail.  

\subsubsection{The continuous-time estimator}  
The following development is inspired  by \cite{989154}. Consider the continuous-time estimator\color{black} 
\begin{equation}  \label{1316} 
\dot{\hat x} = A_c(y,u) \hat x + K_c(y,u)(y - \hat y) + \Psi_c(y,u) \hat\theta + v_c,
\end{equation}
where $\Psi_c(y,u) \hat\theta$ is meant to compensate for the effect of the uncertainty in \eqref{1304}, the term $K_c(y,u)(y - \hat y)$ is common in Luenberger and Kalman-filter based observers, and $v_c$ is a term to be designed. It is assumed that an observer gain $K_c(y,u)$ is known such that the origin for the the dynamical system 
\begin{equation}
  \label{1320} \dot \eta = A_\eta(y,u) \eta, \quad A_\eta:=[A_c(y,u) - K_c(y,u)H]
\end{equation}
is exponentially stable, uniformly in some admissible pairs $(u,y)$.   Indeed,  the state-estimation error dynamics,  resulting from subtracting \eqref{1316} from \eqref{1304} is given by
\begin{equation}  \label{1324} 
\dot e = A_\eta(y,u) e - \Psi_c(y,u) 
\tilde \theta - v_c.
\end{equation}
Hence, it is left to design $v_c$ and $\dot{\tilde\theta} = \dot{\hat\theta}$ such that the second and third terms on the right-hand side of \eqref{1324} vanish. For $\dot{\tilde\theta}$ we seek an implementable  classic adaptation law of gradient type, so we pose 
\begin{equation}  \label{1329} 
  \dot{\hat\theta} := \gamma_c \psi (y - \hat y), \quad \gamma_c >0,
\end{equation}
with $\psi$ to be determined. To render this adaptation law of the ``gradient'' form $\dot{\tilde\theta} = -\gamma_c[\,\cdot\,]\tilde\theta$, we let 
$e := \eta - \Gamma\tilde \theta$ or, equivalently, 
\begin{equation}  \label{1328} 
\eta := e + \Gamma \tilde \theta,
\end{equation}
where $\Gamma$ is to be defined so that $\eta$ satisfies both \eqref{1328} and \eqref{1320}. That is, the latter becomes a target equation for $\eta$ in \eqref{1328}. Now, differentiating on both sides of \eqref{1328}, we obtain
\begin{equation}
  \label{1332} \dot \eta = A_\eta \eta - [A_\eta\Gamma + \Psi_c - \dot \Gamma]\tilde\theta 
+ \gamma_c\Gamma \psi He - v_c,
\end{equation}
in which we temporarily dropped all the arguments to avoid a cumbersome notation.  We see that we recover \eqref{1320} if we set $v_c := \Gamma \psi H e$, which may be implemented as  
  $v_c := \gamma_c\Gamma \psi (y - \hat y)$, 
and $\dot\Gamma :=  A_\eta\Gamma + \Psi_c$.  It is left to define $\psi$ in \eqref{1329}. To that end, we note that the latter is equivalent to 
$\dot{\tilde\theta} = -\gamma_c \psi H\Gamma\tilde\theta + \gamma_c \psi H \eta$,  in which,  $\eta$ is guaranteed to converge to zero exponentially. So, we set  $\psi := \Gamma^\top H^\top$ to obtain  the perturbed gradient-descent system  
\begin{equation}
  \label{1348} \dot{\tilde\theta} = - \gamma_c \psi \psi^\top  \tilde\theta + \gamma_c \psi H \eta. 
\end{equation}
It is intuitively clear that if $\psi \psi^\top$ is PE,  along the system's trajectories, the parameters $\tilde\theta\to 0$. Since also $\eta\to 0$, we obtain that $e\to 0$. This is the rationale that leads to the design of the adaptive observer/identifier for the bilinear system \eqref{1304},  in continuous time,  given by 
\begin{subequations}\label{1355}
\begin{align}  \label{1355a} 
\dot{\hat x} &= A_c \hat x + K_c(y - \hat y) + \Psi_c \hat\theta + \gamma_c\Gamma \psi (y-\hat y),\\ 
\label{1355c} 
\dot{\hat\theta} & = \gamma_c\psi ( y - \hat y), \quad  \psi := \Gamma^\top\! H^\top,    \\   \label{1355b} 
\dot \Gamma & = [A_c - K_cH]\Gamma + \Psi_c, 
\end{align}
\end{subequations}
---cf. \cite{989154,loria2009adaptive}. \\

\subsubsection{The discrete-time estimator}
In discrete time, a similar reasoning leads to the hybrid  estimator given by 
\begin{subequations}\label{1353}
     \begin{align}
         \label{1353a}  \hat x^+ & = A_d \hat x + \left[K_d + \frac{\gamma_d \Gamma^+ \psi}{1 + \gamma_d|\psi|^2}\right] (y - \hat y)  + \Psi_d \hat\theta,  
         \\   \label{1353c} 
 \hat\theta^+ & = \hat \theta + \frac{\gamma_d \psi}{1 + \gamma_d|\psi|^2} (y-\hat y), \quad \gamma_d >0,  \\
         \label{1353b} \Gamma^+ & = [A_d - K_dH]\Gamma + \Psi_d,
     \end{align}
   \end{subequations}
where $K_d$ is designed so that  $\eta^+  = [A_d - K_dH] \eta$ is UES. 

Next, we use the designs above, tailored separately for systems \eqref{1304} and \eqref{1304bis}, to construct a properly defined hybrid estimator for the plant, seen as a time-varying hybrid system. 
To that end, we start by underlining some technical aspects that were left out in Example \ref{example2} for the sake of clarity. \\

\subsubsection{The hybrid estimator}
 Consider the autonomous hybrid plant $\Sigma_P$.  Given an initial condition $x_o \in C \cup D$,  and an input signal $u : \dom u \rightarrow \mathbb{R}^{m_u}$ belonging to a subset of admissible hybrid arcs denoted by $\mathcal{U}$,  we denote by  $\lambda := (x_o,u)$ the vector of variables that  parameterize  the system's solutions. That is,  we denote by $\phi^\lambda$ the solution, which,  without loss of generality,  starts at 
$(t_o,j_o) \in \dom u$.   The solution $\phi^\lambda$ flows and jumps depending on the flow and jump sets $C$ and $D$.   
Furthermore,  we assume that $\dom u = \dom \phi^\lambda$.  There is no loss of generality in this hypothesis because the domain of $u$ can be adjusted by creating virtual jumps,  so that the pair $(u,\phi^\lambda)$ forms a solution pair  to $\Sigma_P$. 

 In turn, the hybrid system produces an output trajectory $(t,j) \to y^\lambda(t,j)$,  given by $y^\lambda(t,j):= H \phi^\lambda(t,j)$, which is  also parameterized by  $\lambda$.   
In addition, we let the set $I_\lambda^j := \left\{t : (t,j) \in \dom \phi^\lambda \right\}$ and, for any function $(y,u)\mapsto M(y,u)$ such that $M\in \{ A_c, K_c, \Psi_c, A_d, K_d, \Psi_d, A_\eta, B_\eta \}$, we define $M^\lambda(t,j) := M(y^\lambda(t,j),u(t,j))$. Thus, after eqs. \eqref{1355} and \eqref{1353}, we introduce  the proposed hybrid observer/identifier: 
\begin{subequations} 
\label{1349}
  \begin{align}
&\hspace{-1.1mm}\dot{\hat x} = A_c^\lambda \hat x + K_c^\lambda(y - \hat y) + \Psi_c^\lambda \hat\theta + \gamma_c \Gamma_c  \psi^{\bar{\lambda}}(y-\hat y),  \quad \\
         \label{1349c} 
&\hspace{-.9mm}\dot{\hat\theta}  = \gamma_c  \Gamma_c^\top \! H^\top  (y - \hat y),  \\
         \label{1349b} 
&\hspace{-2.8mm}\dot \Gamma_c  = A_\eta^\lambda\Gamma_c + \Psi_c^\lambda, \quad \dot \Gamma_d = 0 
\ \qquad
\forall t\,\in \text{int}(I_\lambda^j),  \ 
\end{align}
\end{subequations}
\vskip -15pt
\begin{subequations}\label{1357}
\begin{align}  \label{1357a} 
 \hat x^+ & = A_d^\lambda \hat x + \bigg[K_d^\lambda + \frac{\gamma_d \Gamma_d^+ \psi^{\bar{\lambda}}}{1+\gamma_d|\psi^{\bar{\lambda}}|^2}\bigg](y-\hat y) + \Psi_d^\lambda \hat\theta  ,\\ 
\label{1357c} 
\hat\theta^+ & = \hat\theta + \frac{\gamma_d 
\psi^{\bar{\lambda}}}{1+\gamma_d|\psi^{\bar{\lambda}}|^2} (y - \hat y), \\
\label{1357b} 
\Gamma_c^+ & = \Gamma_c,\quad  \Gamma_d^+ = [I-B_\eta^\lambda]\Gamma_d + \Psi_d^\lambda \mbox{} \hspace{0.2in} \, \text{otherwise},
  \end{align}
   \end{subequations}
 where $A^\lambda_\eta := A^\lambda_c - K^\lambda_c H$,  $B^\lambda_\eta := A^\lambda_d - K^\lambda_d H$, and 
  \begin{align*}
\psi^{\bar\lambda}(t,j) := 
\left\{ 
\begin{array}{cl} 
\Gamma_c(t,j)^\top H^\top & \text{if} ~~  t \in \text{int}(I^\lambda_j) 
\\ 
\Gamma_d(t,j)^\top H^\top & \text{otherwise}.
\end{array}
\right.
\end{align*}

Note that $\psi^{\bar\lambda}$ is parameterized by  
the vector $\bar\lambda$,  which includes $\lambda$ but also includes the initial conditions $(\Gamma_{co}, \Gamma_{do})$ for $(\Gamma_c, \Gamma_d)$; namely,  $\bar{\lambda} := (\lambda, \Gamma_{co}, \Gamma_{do}) $.

\begin{remark}
Note that in assuming that all hybrid arcs have the same hybrid domain $\dom \phi$ 
it is required to know when the system's trajectories $\phi^\lambda$ jump. This is an implicit standing assumption that is needed for a coherent definition of the resulting time-varying hybrid system. However, it is little conservative in that it is not tantamount to assuming the knowledge of unmeasured variables. 
\end{remark}   

The following is our main statement of this section. 
\begin{proposition}\label{prop:OBS}
Consider the hybrid plant 
$\Sigma_P$ in \eqref{SigmaP},  governed  by \eqref{1304} when $x \in C$ and governed by \eqref{1304bis} when $x \in D$.  Assume that $A_d$,  $\Psi_c$,  $\Psi_d$,  $A_c$,  and $A_d$ are continuous, let  $\phi^\lambda$ be a parameterized solution,  such that the corresponding input-output pair $(u, y^\lambda)$ is uniformly bounded (in $\lambda$ and in the hybrid domain) and let   
\begin{equation}
 \label{1418} 
\Sigma^{\bar\lambda}(t,j) := \frac{\gamma_d\psi^{\bar\lambda}(t,j) \psi^{\bar\lambda}(t,j)^\top}{1 + \gamma_d |\psi^{\bar \lambda}(t,j)|^2}.
\end{equation}
  
  Then, the origin $\{(e,\tilde\theta)=(0,0)\}$ is globally exponentially stable, uniformly in $\bar{\lambda}$ if 
\begin{enumerate}[topsep=0pt,label={(\roman{enumi})},leftmargin=*]
\item for each $\lambda$,  the pair $(A_\eta^\lambda,B_\eta^\lambda)$ satisfies Assumption \ref{ass:UGS}; and it is HUO uniformly in $\lambda$;
\item the pair $\left( \psi^{\bar\lambda} {\psi^{\bar\lambda}}^\top, \Sigma^{\bar\lambda} \right)$ is HPE,  uniformly in $\bar{\lambda}$. \\[-24pt]
\end{enumerate}
\end{proposition}

\begin{remark}
In  (i) and (ii) we require HUO and HPE to hold uniformly in $\lambda$ and $\bar \lambda$, respectively.  This means that inequalities \eqref{eq:HUO} and \eqref{eq:HPE} hold with $(K, \mu)$ independent of $(\lambda, \bar\lambda)$.     
\end{remark}

\begin{proof}
The proof follows by invoking Theorems \ref{HUO=>UES} and \ref{thm:PE=>UES}. To show this,  we start by writing the error dynamics in the form \eqref{606}. Using $e := x-\hat x$,  $\tilde\theta := \hat\theta-\theta$, the previously introduced notation, Eqs. \eqref{1304},  \eqref{1304bis}, \eqref{1355}, and  \eqref{1353}, we obtain the estimation error dynamics
\begin{subequations} 
\label{1405}
\begin{align}
 \dot e = &\  A_\eta^\lambda(t,j)e - \Xi_c^{\bar\lambda}(t,j) \tilde \theta 
\label{1405a} 
- \gamma_c\Gamma_c(t,j) \psi^{\bar\lambda}(t,j)H \eta\\   \label{1405b} 
\dot{\tilde\theta} = &\ - \gamma_c \psi^{\bar\lambda}(t,j)\psi^{\bar\lambda}(t,j)^\top \tilde\theta + \gamma_c\psi^{\bar\lambda}(t,j)H \eta \\    \label{1405c}         
\dot \eta = &\  A_\eta^\lambda(t,j)\eta
\end{align}
\end{subequations}
for all $t \in I^j_\lambda$, where to abbreviate we also introduced 
$$ \Xi_c^{\bar \lambda}(t,j):= \left[ \Psi_c^\lambda(t,j) - \gamma_c\Gamma_c(t,j)\psi^{\bar\lambda}(t,j)H\Gamma_c(t,j) \right],  $$ 
and 
\begin{subequations}\label{1399}
\begin{align} \label{1399a} 
e^+ = &\ \big[ I_{m_x} - B_\eta^\lambda(t,j) \big] e - \Xi_d^{\bar\lambda}(t,j) \tilde\theta 
 - \Gamma_d^+(t,j) \Delta^{\bar\lambda}(t,j) \eta \\
         \label{1399b} 
\tilde \theta^+ = &\  \Big[ I_{m_\theta} - \Sigma^{\bar\lambda}(t,j) \Big] \tilde\theta
         + \Delta^{\bar\lambda}(t,j)\eta,\\
         \label{1399c} \eta^+ = &\ \big[ I_{m_x} - B_\eta^\lambda(t,j) \big] \eta
     \end{align}
   \end{subequations}
for all $(t,j) \in \dom \phi^\lambda$ such that $(t,j+1) \in \dom \phi^\lambda$, where 
%
 \begin{align*}
\Xi_d^{\bar\lambda}(t,j) &:= \big[ \Psi_d^\lambda(t,j) - \Gamma_d^+(t,j) \Sigma^{\bar\lambda}(t,j)\big]
\\
 \Delta^{\bar\lambda}(t,j) & :=\frac{\gamma_d 
\psi^{\bar\lambda}(t,j) H}{1 + \gamma_d |\psi^{\bar\lambda}(t,j)|^2}.  
 \end{align*}
The equations \eqref{1405c}-\eqref{1399c} constitute a time-varying system of the form \eqref{606} with $\zeta =\eta$ and $\nu \equiv 0$. Global exponential stability, uniform in  the initial time and  in $\lambda$, follows after Item (i) of the Proposition, invoking Theorem \ref{HUO=>UES}. 

On the other hand,  equations \eqref{1405a}-\eqref{1405b} together with \eqref{1399a}-\eqref{1399b} form, in turn, another system of the form \eqref{606}, with state $\zeta := [e^\top\ \tilde\theta^\top]^\top$ and parameterized input 
\begin{equation}  \label{1436} 
\hspace{-0.4cm} \nu^{\bar\lambda}(t,j) := \left\{\hspace{-1ex}
\begin{array}{ll}
  \begin{bmatrix}
     -\gamma_c\Gamma_c(t,j) \psi^{\bar\lambda}(t,j)H\eta(t,j) \\
     \gamma_c \psi^{\bar\lambda}(t,j)H\eta(t,j)
  \end{bmatrix} \quad &  \hspace{-0.6cm} t \in \text{int}(I^j_\lambda)
  \\[10pt]
\begin{bmatrix}
  \Gamma^+_d(t,j) \Delta^{\bar\lambda}(t,j)\eta(t,j) \\
  \Delta^{\bar\lambda}(t,j) \eta(t,j)
\end{bmatrix}  & \hspace{-0.6cm}   \mbox{otherwise}.
\end{array}\right.
\end{equation}
Since $\eta$ converges uniformly and exponentially to zero and the factors of $\eta(t,j)$ in \eqref{1436} are uniformly bounded (both  in the initial time and in $\lambda$),  it is only left to show that the system \eqref{1405a}-\eqref{1405b} together with \eqref{1399a}-\eqref{1399b} is ISS with respect to $\nu^{\bar\lambda}$,  as in \eqref{1436}.  After the proof of Theorem \ref{HUO=>UES}, ISS follows if the origin is UES for the system with $\nu^{\bar\lambda} \equiv 0$,  which is governed by 
\begin{subequations} \label{1441}
\begin{align} \label{1441a} 
\dot e = &\  A_\eta^\lambda(t,j)e - \Xi_c^{\bar\lambda}(t,j) \tilde \theta\\         \label{1441b} 
\dot{\tilde\theta} = &\ -\gamma_c\psi^{\bar\lambda}(t,j)\psi^{\bar\lambda}(t,j)^\top \tilde\theta
 \end{align}
\end{subequations}
for all $t \in I^j_\lambda$, and 
\begin{subequations}\label{1448}
     \begin{align}
         \label{1448a} e^+ = &\ \big[ I_{m_x} - B_\eta^\lambda(t,j) \big] e - \Xi_d^{\bar\lambda}(t,j)\tilde\theta \\
         \label{1448b} 
\tilde \theta^+ = &\  \left[ I_{m_\theta} - \frac{\gamma_d \psi^{\bar\lambda}(t,j) \psi^{\bar\lambda}(t,j)^\top}{1 + \gamma_d |\psi^{\bar\lambda}(t,j)|^2} \right] \tilde\theta,
\end{align}
\end{subequations}
 for all $(t,j) \in \dom \phi^\lambda$ such that $(t,j+1) \in \dom \phi^\lambda$.   To show this, we invoke again Theorem \ref{HUO=>UES} under Items (i) and (ii) of the proposition. This time, we regard  the equations \eqref{1405a}-\eqref{1399a} as a time-varying hybrid system of the form \eqref{606} with state $\zeta = e$ and parameterized input $\nu^{\bar\lambda}$ defined as 
\begin{equation}
  \label{1422} 
\nu^{\bar\lambda}(t,j) = \left\{
\begin{array}{ll}
- \Xi_c^{\bar\lambda}(t,j) \tilde \theta(t,j) & t \in \text{int}(I^j_\lambda) \\[3pt]
 - \Xi_d^{\bar\lambda}(t,j) \tilde\theta(t,j) &   \mbox{otherwise}.
\end{array}\right.
\end{equation}
Note that the input defined in \eqref{1422} corresponds to uniformly bounded functions in factor of the solution to \eqref{1441b}-\eqref{1448b}, while the origin $\{e = 0\}$ is UES for the system 
\begin{align*}
  \dot e = &\  A_\eta^\lambda(t,j)e & t \in\text{int}(I^j_{\lambda})  \\
  e^+ = &\ \big[ I_{m_x} - B_\eta^\lambda(t,j) \big] e & (t,j),(t,j+1) \in \dom \phi^\lambda,
\end{align*}
as already established since, by assumption, for this system $A_\eta^\lambda$ and $B_\eta^\lambda$ satisfy Assumptions \ref{ass:UGS}--\ref{ass:linf} and the HUO property.  Furthermore, under the same conditions, \eqref{1441}-\eqref{1448} is ISS with respect to $\nu^{\bar\lambda}$ in \eqref{1422}.   It is only left to establish UES of $\{\tilde \theta = 0\}$ for \eqref{1441b}-\eqref{1448b}, which has exactly the form of the gradient-descent error dynamics studied in the previous section---see \eqref{1418}.  Therefore,  UES for \eqref{1441b}-\eqref{1448b} follows provided that the pair $\left( \psi^{\bar\lambda} {\psi^{\bar\lambda}}^\top, \Sigma^{\bar\lambda} \right)$ is HPE, uniformly in $\bar\lambda$,  that is Item (ii) of the proposition.
\end{proof}

\subsubsection{A numerical example} 
We illustrate the performance of the hybrid observer/identifier designed above on the vertical bouncing-ball system with actuated jumps given by
\begin{equation*}
\begin{aligned} 
\Sigma_p : & \left\{ 
\begin{matrix} 
\dot{x}=\begin{bmatrix} x_2 & - \theta \end{bmatrix}^\top & x  \in  \mathbb{R}_{\geq 0} \times \mathbb{R}
\\[3pt]
x^+=\begin{bmatrix} 0 \\ - 
\frac{\theta}{12.2625} x_2 + u \end{bmatrix} &  x \in  \left\{ x_1 = 0,~x_2 \leq 0 \right\}, 
\end{matrix} \right.
\end{aligned}
\end{equation*}
where $x \in \mathbb{R}^{2}$ includes the ball position and velocity,  $u \in \mathbb{R}$ is the input,  $y = x_1$ is the output,  and $\theta \in \mathbb{R}$ is a constant unknown parameter. The constant parameter $\theta$ in both the flow and the jump dynamics is not physically motivated, but it allows to illustrate the importance of HPE.

The dynamical model of the considered example is of the type as in \eqref{1304}-\eqref{1304bis}, with $A_c=\begin{bmatrix} 0 & 1 \\ 0 & -0.1 \\ \end{bmatrix}, \quad A_d=\begin{bmatrix} -1 & 0 \\ 0 & 0 \\ \end{bmatrix},$ $H = [ 1 \quad 0 ]$, $B_c = 0$,   $B_d=\begin{bmatrix} 0 & 1 \end{bmatrix}^\top$, $\Psi_c=\begin{bmatrix} 0 & -1  \end{bmatrix}^\top$, and $\Psi_d=\begin{bmatrix} 0 & -\frac{x_2}{12.2625}  \end{bmatrix}^\top$,  

The objective is to jointly estimate the state $x_2$ and the unknown parameter $\theta$ using the measurement of $y=x_1$ and the knowledge of the system's structure.  We assume that we can detect instantaneously when the solution of the system jumps.  Then, following Proposition \ref{prop:OBS}, we design the observer gains so as to satisfy Assumption \ref{ass:UGS}. 
First, we find scalars $a_c$,  $a_d < 0$,  $\beta$,  $M > 0$, matrices $K_c, K_d \in \mathbb{R}^{m_x \times m_{y}}$,  and a positive definite symmetric matrix $P \in \mathbb{R}^{m_x \times m_x}$ such that
\begin{subequations}
\label{eqn:Lyapcom1} 
\begin{align}
 (A_c - K_c H)^{\top} P + P (A_c - K_c H) & \leq a_c P
\\
(A_d - K_d H)^{\top} P (A_d - K_d H) & \leq e^{a_d} P
\\
  a_c t+a_d j \leq M-\beta(t+j)  \qquad \forall (t,j) \in &\mathbb{R}_{\geq 0} \times \mathbb{N}_{\geq 0}.
\end{align}
\end{subequations}
Indeed, under inequalities (\ref{eqn:Lyapcom1}) Assumption~\ref{ass:UGS} holds with $Q_d=(e^{a_d}-1)I_{m_{\theta}}$. Then, using the Schur complement, condition (\ref{eqn:Lyapcom1})  boils down to solving the LMI
\begin{align*}
    A_c^{\top} P + P A_c - L_c H - H^{\top} L_c^{\top} & <  0,  
     \\
    \begin{bmatrix} 
    P & (PA_d-L_dH)^{\top}  \\ PA_d-L_dH & P
    \end{bmatrix} & >  0,
\end{align*}
 which can be done using the toolbox YALMIP~\cite{lofberg2004yalmip}.  
The variables of the LMIs are $P \in \mathbb{R}^{2 \times 2}$ and $L_c, L_d \in \mathbb{R}^2$, and the observation gains are set to $K_c=P^{-1}L_c$ and $K_d=P^{-1}L_d$. 

Thus, we implemented the hybrid adaptive observer \eqref{1349}-\eqref{1357} and performed several illustrative simulations, with and without HPE. The results are showed in Figures \ref{fig:persis_No}--\ref{fig:persis_hybrid}.
\begin{figure}[h!]
	\begin{center}
		\includegraphics[width=85mm]{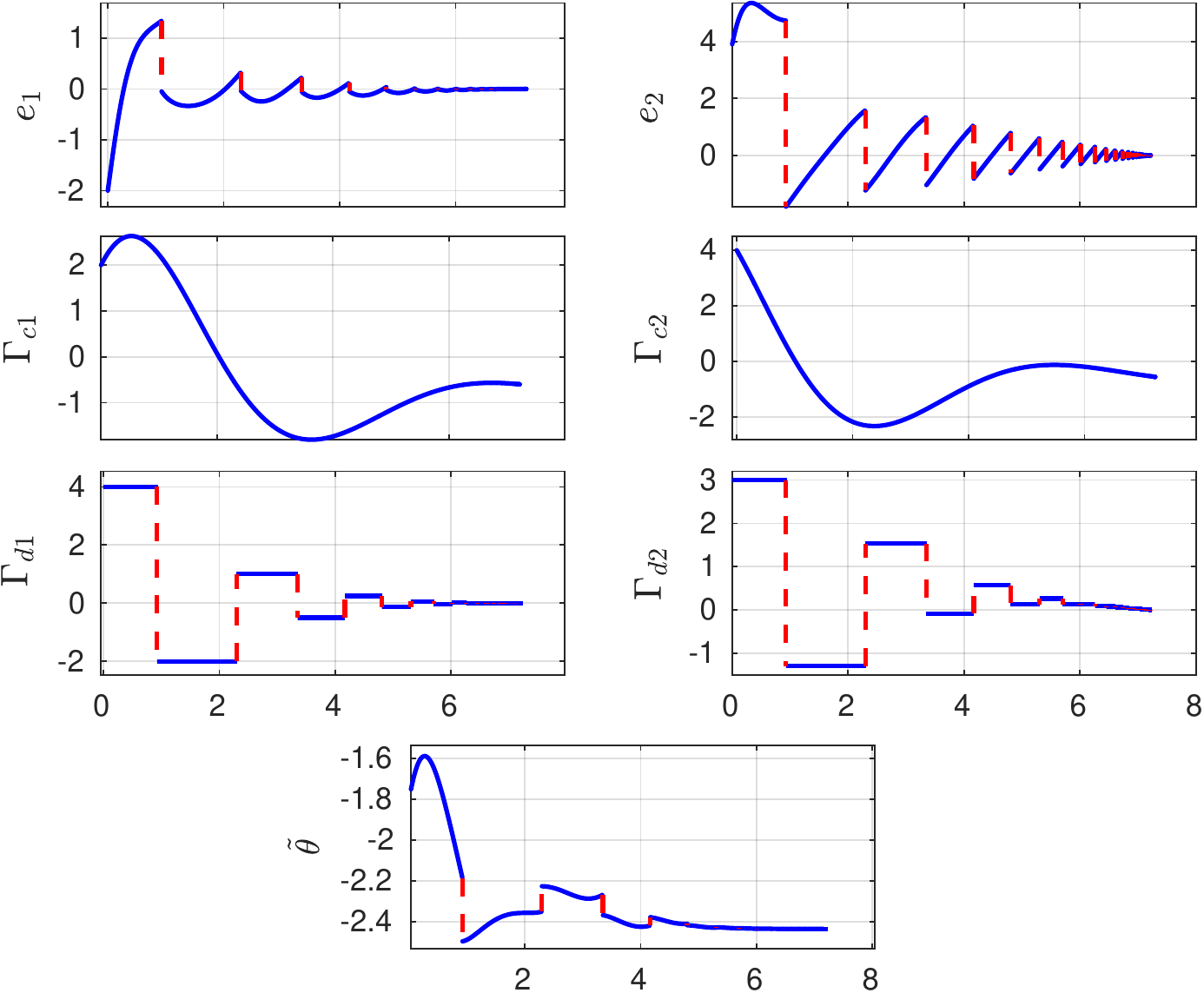}
	\end{center}
	\caption{Performance of the adaptive observer with $u=0$, {\it i.e.}, without HPE. The estimation errors $|\tilde\theta|$ fail to vanish}
	\label{fig:persis_No}	
\end{figure}

We show the evolution of the state estimation errors $e_i$, the filtered regressors $\Gamma_c$ and  $\Gamma_d$, and the parameter estimation error $\hat{\theta}-\theta$. In all the figures, the solid blue lines represent the flows and dashed red lines represent the jumps. 

The adaptation gains are set to $\gamma_c=0.4$ and $\gamma_d=0.8$ and for the observation gains we use the toolbox YALMIP, to find $K_c=[0.7215\quad  1.1184]^{\top}$ and $K_d=[-0.5\quad  0.5]^{\top}$, which satisfy the conditions from Proposition \ref{prop:OBS}. The initial conditions were set to $\hat{x}(0,0)=[4\quad 0.1]^{\top}$, $\Gamma_c(0,0)=[2\quad  4]^{\top}$, $\Gamma_d(0,0)=[4\quad  3]^{\top}$, and $\hat{\theta}(0,0)=8$.

First, we set $u=0$, so the system lacks excitation; the results are shown in Figure~\ref{fig:persis_No}. It is showed that the state estimation errors converge (exponentially), but the parameter estimate $\hat{\theta}$ does not converge to $\theta$. This is due to the fact that the pair $ \left( \psi^{\bar\lambda} {\psi^{\bar\lambda}}^\top,    \Sigma^{\bar\lambda} \right)$ is not HPE.  
\begin{figure}[h!]
	\begin{center}
		\includegraphics[width=85mm]{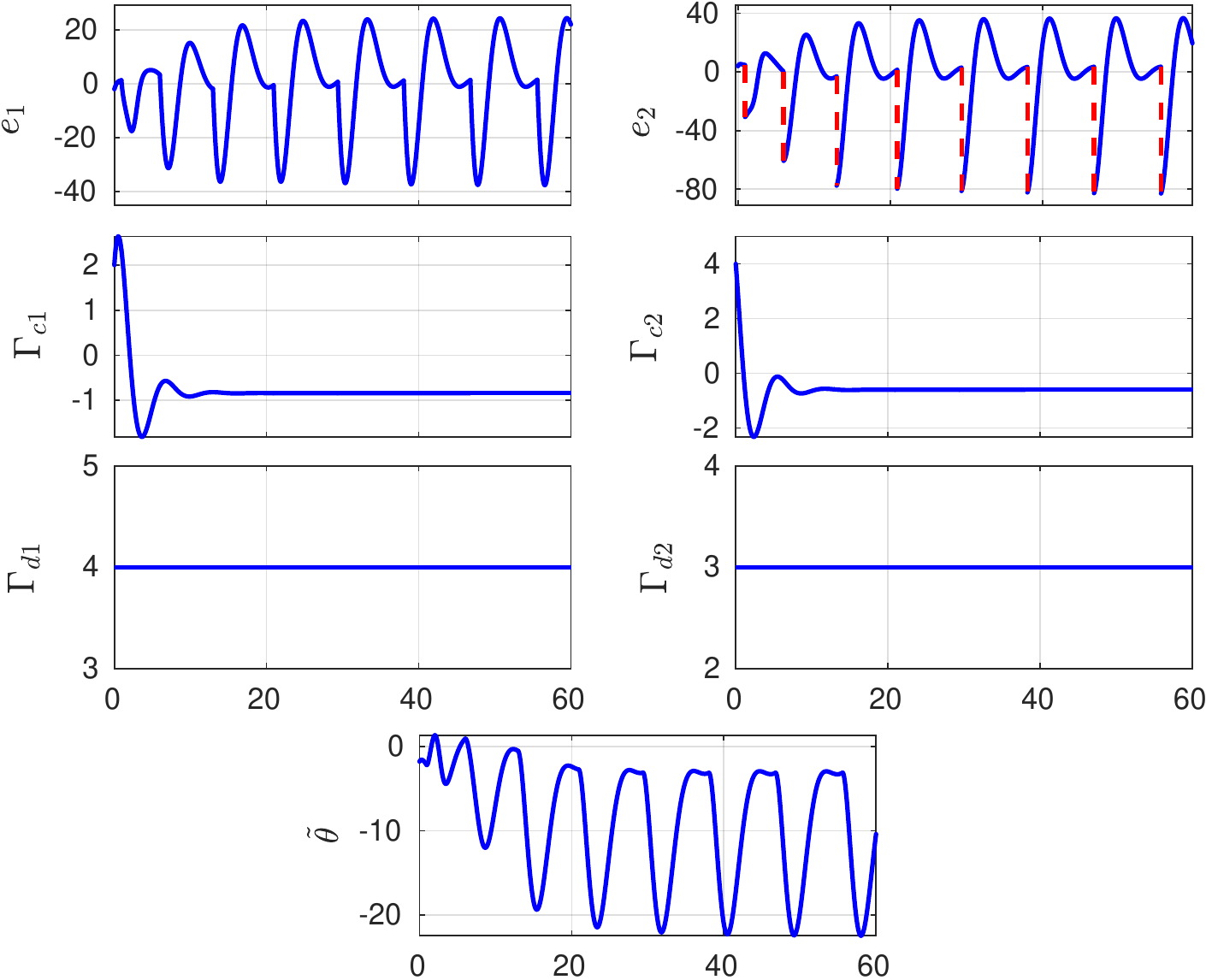}\\
	\end{center}
	\caption{Performance of the continuous-time adaptive observer/identifier from \cite{989154} with persistent input $u \equiv 20$. }
	\label{fig:persis_flow}	
\end{figure}
\begin{figure}[h!]
	\begin{center}
		\includegraphics[width=87mm]{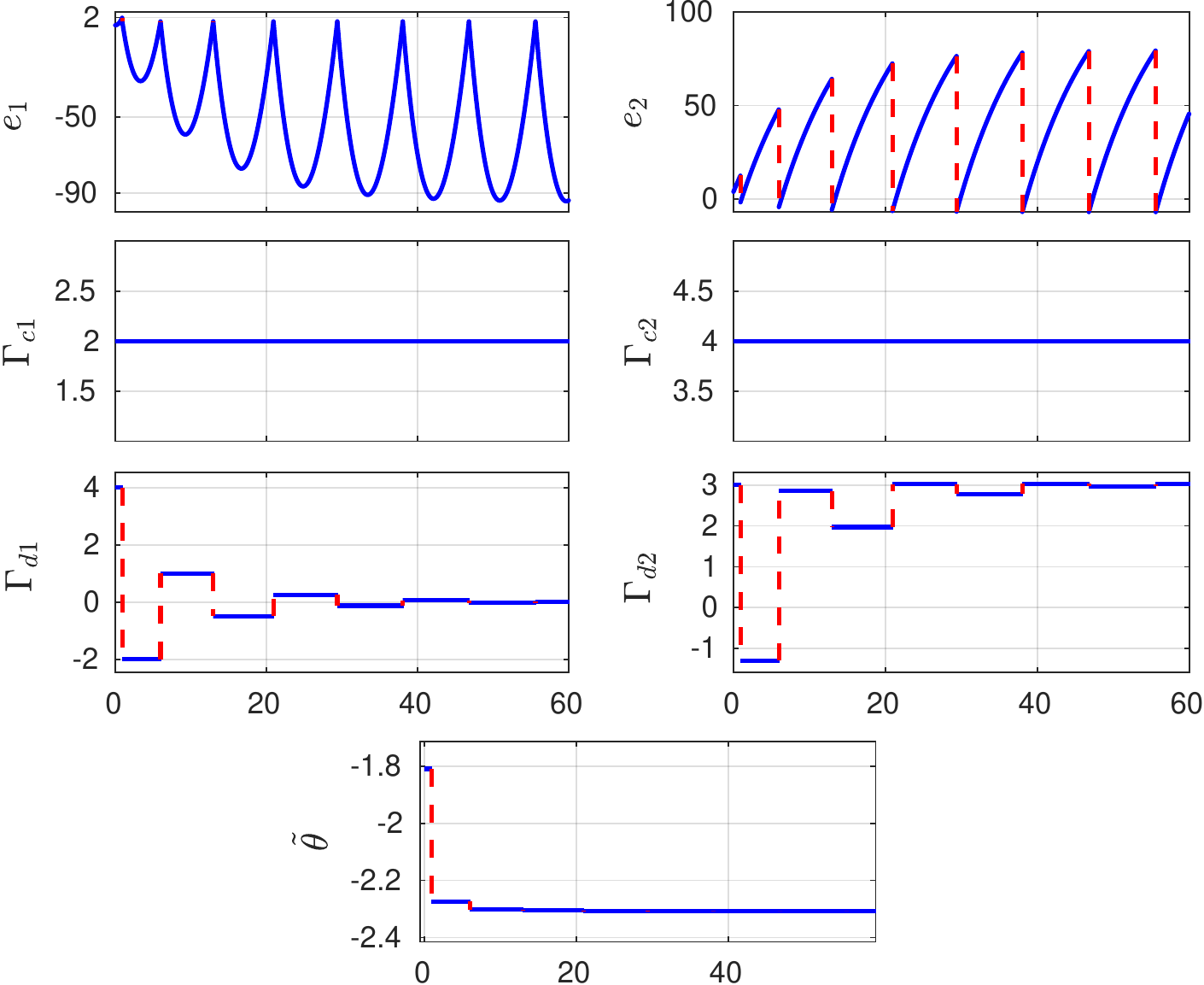}\\
	\end{center}
	\caption{Performance of the discrete-time adaptive observer/identifier from \cite{guyader2003adaptive} with $u \equiv 20$.}
	\label{fig:persis_jump}	
\end{figure}

In two other runs of simulation, we set $u\equiv 20$ and used the purely continuous-time adaptive observer from~\cite{989154}---see the results in Figure~\ref{fig:persis_flow}, and the purely discrete-time adaptive observer from~\cite{guyader2003adaptive}---the results are showed in Figure~\ref{fig:persis_jump}. In both  cases, neither the state-estimation nor the parameter-estimation errors vanish.

Finally, in a fourth simulation we tested the hybrid adaptive observer/identifier under the same conditions. In Figure~\ref{fig:persis_hybrid},  one can appreciate that both the state- and parameter-estimation errors vanish.  Indeed, in this case,  the pair $\left( \psi^{\bar\lambda} {\psi^{\bar\lambda}}^\top,  \Sigma^{\bar\lambda} \right)$ is HPE with $K=2$ and $\mu=0.7$. 

\begin{figure}[h!]
	\begin{center}
		\includegraphics[width=85mm]{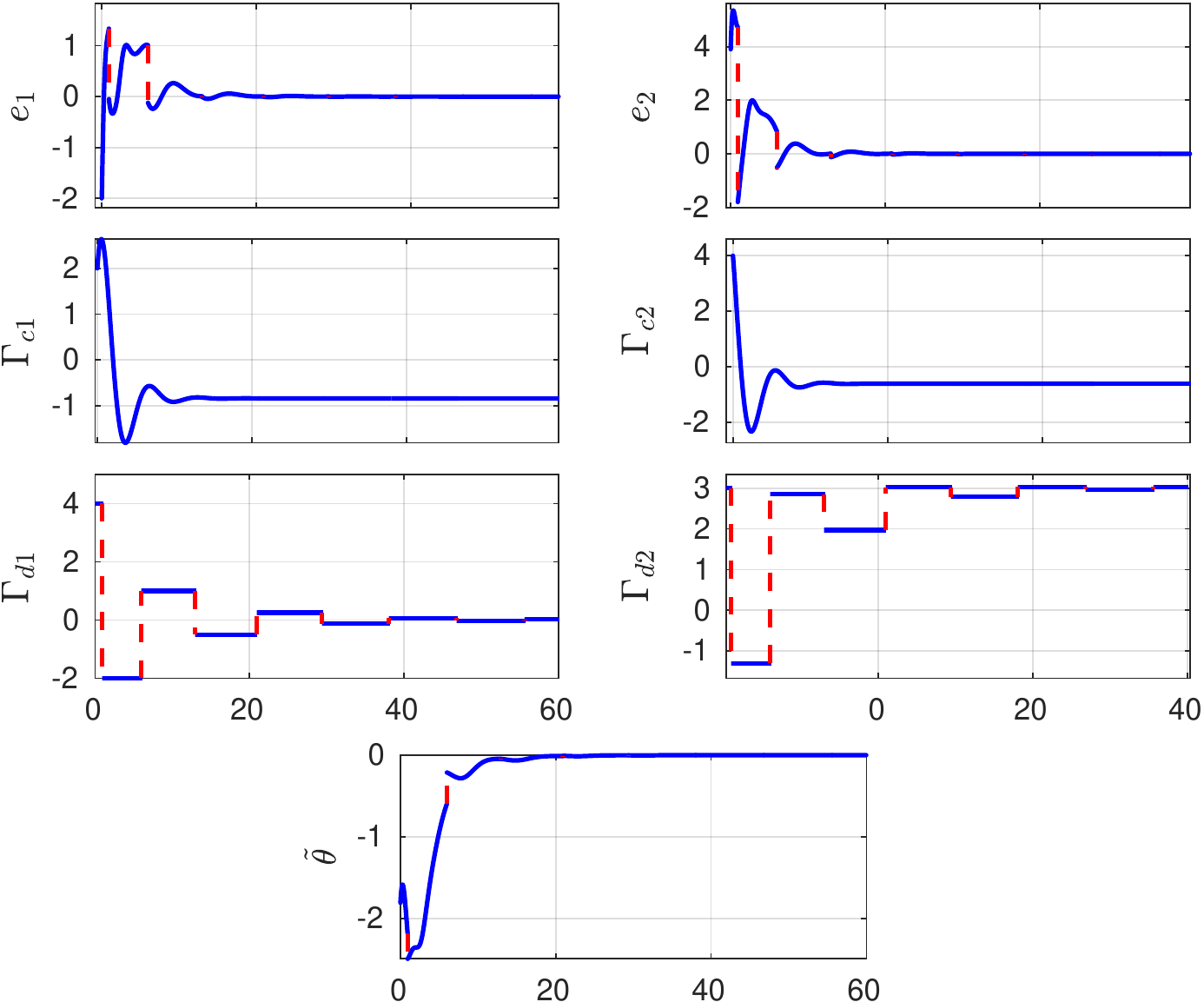}\\
	\end{center}
	\caption{Performance of the hybrid adaptive observer with $u \equiv 20$.}
	\label{fig:persis_hybrid}	
\end{figure}  

\clearpage
\section{Conclusion and Future Work}

This paper generalized some stability and robustness properties of linear time-varying systems,  encountered in estimation theory,  to the more general context of hybrid systems.   By introducing the class of linear (non-autonomous) hybrid systems in $\mathcal{H}_u$, we showed that a relaxed (hybrid) version of the well-known PE condition is sufficient to guarantee UES as well as ISS.   The proposed hybrid framework applies to the estimation problem when a linear input-output regression model is fed with hybrid data.  Furthermore,  it allows the design and analysis of adaptive observers/identifiers for a class of uncertain hybrid systems capable of tracking  both  the state and the unknown parameters.   For future work,   while we assumed the jumps of the estimation algorithm to be synchronized with the jumps of the hybrid regressor,  this condition may be unrealistic in practice,  since the regressor's jumps cannot always be detected instantaneously.  Hence,  robustness of the proposed approach with respect to delays in the  jumps detection could be analyzed along the lines of~\cite{altin2019hybrid}.

\subsection*{Appendix: Hybrid Comparison Lemma}

We introduce the following comparison lemma for hybrid systems,  which is a particular case of   \cite[Lemma C.1]{cai2009characterizations}.
\begin{lemma}
\label{lem:comparison}
Consider a hybrid arc $v : \dom v \rightarrow \mathbb{R}_{\geq 0} $ and assume the existence of $a$, $b>0$ such that
\begin{itemize}
    \item For all $(t,j) \in \dom v$ such that $(t,j+1) \notin \dom v$, $$\dot{v}(t,j) \leq -a v(t,j) + b.  $$ 
   \item For all $(t,j) \in \dom v$ such that $(t,j+1) \in \dom v$, $$v(t,j+1)-v(t,j) \leq -a v(t,j) + b.  $$
\end{itemize}
Then, there exists $c> 0$ such that $v(t,j) \leq e^{-a(t+j)} v(0,0) + c b$ for all $(t,j) \in \dom v$.
\end{lemma}
The proof follows the same steps as the proof of~\cite[Lemma C.1]{cai2009characterizations} for the particular case where the hybrid arc $\alpha(t,j)=a$ is constant,  and for which,  the map $\gamma_{\alpha}:\mathbb{R}_{\geq 0} \times \dom v$ is explicitly given by $\gamma_{\alpha}(r,t,j)= e^{-a(t+j)}r$.

\def\loria{Lor\'{\i}a}
  \def\nesic{Ne\v{s}i\'{c}\,}\def\nonumero{\def\numerodeitem{}}


\end{document}